\documentclass[smallextended]{svjour3}

\usepackage[utf8]{inputenc}

\usepackage{amsmath}
\usepackage{amsfonts}
\usepackage{amssymb}
\usepackage{mathrsfs}
\usepackage{dsfont}

\usepackage{amstext, amsopn}

\usepackage{mathtools}

\usepackage[lined,ruled, noend]{algorithm2e}

\usepackage{listings}

\usepackage{cite}
\usepackage{enumerate}
\usepackage{fancybox}

\RequirePackage{hyperref}
\PassOptionsToPackage{pdftex,bookmarksopen,bookmarksnumbered}{hyperref}
\usepackage{hyperref}

\usepackage[colorinlistoftodos]{todonotes}

\usepackage{fullpage}

\usepackage{amsthm} 
\newtheorem{thm}{Theorem}

\newtheorem{lem}[thm]{Lemma}
\newtheorem{cor}[thm]{Corollary}
\newtheorem{propn}[thm]{Proposition}
\newtheorem{prop}[thm]{Proposition}
\newtheorem{assumption}{Assumption}

\newtheorem{eg}{Example}

\newcommand{\Ccal}{{\mathcal C}}

\newcommand{\Ecal}{{\mathcal E}}

\newcommand{\Ical}{{\mathcal I}}

\newcommand{\Pcal}{{\mathcal P}}

\newcommand{\Ebb}{\mathbb{E}}

\newcommand{\Ibb}{\mathbb{I}}

\newcommand{\Nbb}{\mathbb{N}}

\newcommand{\Pbb}{\mathbb{P}}

\newcommand{\Rbb}{\mathbb{R}}

\newcommand{\Lbar}{{\overline{L}}}

\newcommand{\pbar}{{\overline{p}}}

\newcommand{\tbar}{{\overline{t}}}

\newcommand{\xbar}{{\overline{x}}}

\newcommand{\alphabar}{{\overline{\alpha}}}

\newcommand{\mubar}{{\overline{\mu}}}

\newcommand{\epsilonbar}{{\overline{\epsilon}}}

\newcommand{\pbf}{{\bf{p}}}

\newcommand{\paren}[1]{\left(#1\right)}
\newcommand{\ecklam}[1]{\left[#1\right]}

\renewcommand{\equiv}{:=}

\newcommand{\Rn}{{\Rbb^n}}

\newcommand{\Rtw}{{\Rbb^2}}
\newcommand{\extre}{(-\infty,+\infty]}

\newcommand{\ucmin}[2]{\underset{#2}{\mbox{minimize}}~#1}

\newcommand{\set}[2]{\left\{#1\,\left|\,#2\right.\right\}}
\newcommand{\map}[3]{#1:\,#2\rightarrow #3\,}
\newcommand{\mmap}[3]{#1:\,#2\rightrightarrows #3\,}

\newcommand{\ip}[2]{\left\langle #1,~ #2\right\rangle}

\newcommand{\norm}[1]{\left\|#1\right\|}

\newcommand{\cpr}[2]{\ensuremath{\mathbb{P}\left(#1\,\middle|\,#2\right)}}

\newcommand{\und}{\quad\mbox{and}\quad}
\newcommand{\where}{\quad\mbox{ where }\quad}
\newcommand{\fur}{\mbox{ for }}

\DeclareMathOperator{\dd}{d}
\DeclareMathOperator{\inv}{inv}

\DeclareMathOperator{\Id}{Id}

\DeclareMathOperator{\prox}{prox}

\DeclareMathOperator{\argmin}{argmin\,}
\DeclareMathOperator{\crit}{crit\,}

\DeclareMathOperator{\Supp}{supp} 

\DeclareMathOperator{\Fix}{\mathsf{Fix}\,}
\DeclareMathOperator{\Find}{\mathsf{Find}\,}

\newcommand{\sd}{\partial}

\title{Convergence in Distribution of Randomized Algorithms:\\ The Case of Partially Separable Optimization}
\author{
{D. Russell Luke} 
\thanks{Institute for Numerical and Applied Mathematics,
    University of Goettingen,
    37083 Goettingen, Germany. DRL was supported in part by 
    the Deutsche Forschungsgemeinschaft (DFG, German Research Foundation) – Project-ID 432680300 – SFB 1456.
	\texttt{r.luke@math.uni-goettingen.de}; ORCID 0000-0002-4508-7360.}
}

\date{\today}

\begin{document}
 \maketitle

 \begin{abstract}
We present a Markov-chain analysis of blockwise-stochastic algorithms for solving partially
block-separable optimization problems.  Our main contributions to the extensive literature 
on these methods are statements about the  
Markov operators and distributions behind the iterates of stochastic algorithms, 
and in particular the regularity of Markov operators and rates of convergence of the 
{\em distributions} of the corresponding Markov chains.  This provides a detailed characterization 
of the moments of the sequences beyond just the expected behavior.  
This also serves as a case study of how randomization restores favorable 
properties to algorithms that iterations of only partial information destroys.  
We demonstrate this on stochastic blockwise implementations of the 
forward-backward and Douglas-Rachford algorithms for nonconvex 
(and, as a special case, convex), nonsmooth optimization.  
\end{abstract}

{\small \noindent {\bfseries 2010 Mathematics Subject Classification:}
  Primary 
  65C40, 
  90C06, 
  90C26;  
    Secondary  
    46N30, 
    60J05, 
    49M27, 
    65K05.\\ 
  }


\section{Introduction}
We present a Markov-chain analysis of blockwise-stochastic algorithms for solving 
\begin{equation}\label{e:P1}
 \ucmin{f(x)+ \sum_{j=1}^m g_j(x)}{x\in\Ecal}.
\end{equation}
Here $\Ecal$ is a Euclidean space that is decomposed into a direct sum of the 
subspaces $\Ecal_j$, denoted $\Ecal=\bigoplus_{j=1}^m \Ecal_j$, 
and for each $j=1,2,\dots,m$, the function $f$ is continuously differentiable with 
blockwise-Lipschitz gradients, $g_i$ is everywhere subdifferentially regular 
(the regular and limiting subgradients coincide) and  
\begin{equation}\label{e:gi}
 g_j(x) = h_j(x_j)
\end{equation}
for $\map{h_j}{\Ecal_j}{\extre}$ subdifferentially regular. 
This represents a partially separable structured optimization problem.  

Problems with this structure are ubiquitous, and particular attention has focused on 
iterative algorithms for large-scale instances where the iterates are generated from 
only partial evaluation of the objective.  Which partial information to access in each iteration 
is randomly selected and computations can be done in parallel across distributed 
systems \cite{Wri15, RichTak14, Richtarik16, PesAud15}.  There is a rich literature on the 
analysis of these methods, focusing mainly
on deterministic properties of the objective function and expectations, iteration complexity, 
convergence of objective values, and acceleration strategies 
\cite{Nes12,RichTak14,fercoq2015accelerated, Pesquet19, ComPes15, LuXia15, NecCli16, Nedic2011, RichQu16}.  
Our own contributions 
to the literature on such stochastic 
methods has focused on a stochastic block-coordinate primal-dual method for the instance of 
\eqref{e:P1} where $f(x)$ is the indicator function of an affine subspace \cite{LukMal18}.
We will touch on primal-dual approaches via a stochastic blockwise  Douglas-Rachford Algorithm 
\ref{algo:sbdr}, but more practical primal-dual approaches to nonsmooth problems are not 
on the agenda of the present study.  

Our main contributions to the extensive literature on these methods are statements about the 
Markov operators and distributions behind the iterates of stochastic algorithms in the most complete sense 
possible.  By that we mean not only statements about the limits of the ergodic sequences,
which only tell one about the expectation, but rather the limiting distributions of the sequence
of measures behind the iterates, when viewed as a Markov chain (see Theorem \ref{t:msr convergence}
and Proposition \ref{t:convergence sbfb-sbdr}).  This allows one to access 
the moments of the limiting sequence, not just its mean.  

Getting a handle on the distributions behind iterates of randomized algorithms 
is significant not only for its generality, but also for the range of practical applications this 
encompasses.  To explain this we note that, in its most general form, 
{\em consistency} of the update functions generating the Markov operators is not assumed.  
In plain terms, the update functions in the Markov chain need not have common fixed points.
To see why this matters, it is first important to recognize that the literature on 
randomized algorithms is exclusively concerned with {\em almost sure} convergence.
In \cite[Proposition 2.5]{HerLukStu22a} it is shown that almost sure convergence of the 
iterates of such Markov chains can only happen when the update functions have common fixed points.  
Situations where the update functions do not have common fixed points are only 
a small perturbation away:  consider any fixed point iteration with numerical error.  
To be sure, the consistent case allows for tremendous simplifications, and we show this 
in sections \ref{s:consistent} and \ref{s:consistent2};  the point is, however, that 
our approach goes far beyond this idealized case.  

Previous work has established a foundation for 
this based on a fixed point theoretic approach 
\cite{ButnariuFlam95, Butnariu95, ButnariuCensorReich97, HerLukStu19a, HerLukStu22a, HerLukStu22b}.  
A different perspective, modeled after a more direct analysis of the {\em descent} properties
of algorithms in an optimization context has been established recently by Salzo and Villa 
\cite{salzo2021}.  This was further developed in the masters thesis 
of Kartamyschew \cite{Kartamyschew}.  In the present work we extend the results of 
\cite{Kartamyschew} to a fully nonconvex setting for more general mappings.  

A noteworthy feature of blockwise methods, and what distinguishes the present study from 
\cite{HerLukStu19a, HerLukStu22a, HerLukStu22b} is that, even when the objective in 
\eqref{e:P1} is convex, blockwise algorithms
do {\em not} satisfy the usual regularity properties enjoyed by convex optimization algorithms that 
lead generically to global convergence.  This is demonstrated in Example \ref{eg:T_i not afne}.  
The stochastic implementations for convex problems, however, 
{\em do} enjoy nice properties {\em in expectation} (see Theorem \ref{t:afneie}), 
and this is enough to guarantee generic global 
convergence (Theorem \ref{t:msr convergence}, Proposition \ref{t:convergence sbfb-sbdr}).  
While this fact lies implicitly behind the convergence analysis 
of, for instance, \cite{LukMal18} and many others, it was recognized in \cite{salzo2021} as the 
important property of {\em descent in expectation}. We place these observations in the 
context of Markov operators with update functions that satisfy desirable properties 
in expectation (see Theorem \ref{t:harvest1}).  These notions, at the level of the Markov
operator, have already been defined in \cite{HerLukStu19a, HerLukStu22a, HerLukStu22b};  the 
convergence results presented in those works, however, are based on the assumption that 
each of the update functions that generate the Markov operator have the same {\em class} of 
regularity that they have in expectation.  Blockwise algorithms for partially 
separable optimization do not enjoy this structure, and therefore many of the results of 
\cite{HerLukStu19a, HerLukStu22a, HerLukStu22b} do not immediately apply;  indeed, 
we conjecture that some of the stronger convergence results of \cite{HerLukStu19a, HerLukStu22a}
are not true without additional compactness assumptions, hence our analogous global 
convergence statement for the convex case Proposition \ref{t:asymp reg}, is weaker  
than its counterparts \cite[Theorem 3.6]{HerLukStu19a} or 
\cite[Theorem 2.9]{HerLukStu22a}.  

The basic machinery of stochastic blockwise  function iterations (Algorithm \ref{algo:sbi}) 
and Markov chains is reviewed 
in section \ref{s:Markov}.  In section \ref{s:regularity} we review and establish
the chain of regularity lifted from the regularity of the individual mappings
on the sample space, Theorem \ref{t:afneie}, to the regularity of 
the corresponding Markov operators on the space of probability measures, Theorem \ref{t:harvest1}.
In section \ref{s:consistent} the special case of consistent stochastic feasibility is 
detailed, showing in particular how the abstract objects for the general case simplify
(see Theorem \ref{t:invMeasforParacontra}). In section \ref{s:convergence} we present 
abstract convergence results, with and without rates (Proposition \ref{t:asymp reg}
and Theorem \ref{t:msr convergence}).  The key to quantitative results in the space of 
probability measures is {\em metric subregularity} of the {\em invariant Markov transport discrepancy}
\eqref{eq:Psi}. This is shown in the case of consistent stochastic feasibility to be {\em necessary}
for quantitative convergence of {\em paracontractive} Markov operators in Theorem \ref{t:msr necessary}.

We return
to the specialization of stochastic partial blockwise splitting algorithms 
in section \ref{s:sbs}, where we develop a case study of stochastic blockwise forward-backward
splitting (Algorithm \ref{algo:sblfb}) and 
stochastic blockwise Douglas-Rachford (Algorithm \ref{algo:sbdr}), establishing the 
regularity of the corresponding fixed point operators (Propositions \ref{t:f_j}-\ref{t:fb-dr aafne})
and convergence in distribution of the corresponding Markov chains (Proposition \ref{t:convergence sbfb-sbdr}).  

\section{Notation and Random Function Iterations}\label{s:Markov}
As usual, $\mathbb{N}$ denotes the natural
numbers including $0$.  
We denote by $\mathscr{P}(G)$ the set of all
probability measures on $G\subset \Ecal$; the measurable 
sets are given by the 
Borel sigma algebra on a subset  $G\subset\Ecal$, denoted by $\mathcal{B}(G)$.     
The notation $X \sim \mu\in \mathscr{P}(G)$ means that the law
 of $X$, denoted $\mathcal{L}(X)$, satisfies 
 $\mathcal{L}(X)\equiv\mathbb{P}^{X} :=\mathbb{P}(X \in \cdot) = \mu$, 
 where $\mathbb{P}$ is the probability
 measure on some underlying probability space.
The open ball centered at $x\in\Ecal$ with radius $r>0$ is denoted $\mathbb{B}(x,r)$;
 the closure of the ball is denoted
$\overline{\mathbb{B}}(x,r)$.  
The distance of a point $x\in \Ecal$ to a set 
$A\subset \Ecal$ in the metric $d$ is denoted by $d(x,A)\equiv \inf_{w\in A}d(x,w)$. 
The {\em projector} onto a set $A$ is denoted by $P_A$ and $P_A(x)$ is the set of 
all points where $d(x,A)$ is attained.  This is empty if $A$ is open, and a singleton 
if $A$ is closed and convex; generically, 
$P_A$ is a (possibly empty) set-valued mapping, 
for which we use the notation $\mmap{P_A}{\Ecal}{\Ecal}$.  
For the ball of radius $r$ around a
subset of points $A\subset \Ecal$, we write
$\mathbb{B}(A,r)\equiv \bigcup_{x\in A} \mathbb{B}(x,r)$.  

Let $\Ibb$ denote an index set, each element $i\in \Ibb$ of which 
is a unique assignment to nonempty subsets
of $\{1,2,\dots,m\}$:  $M_i\in 2^{\{1,2,\dots,m\}}\setminus\emptyset$ for $i\in\Ibb$, 
where $\cup_{i\in\Ibb}M_i = 2^{\{1,2,\dots,m\}}\setminus\emptyset$ and $M_i\neq M_j$
for $i\neq j$.  For convenience we will let the first such subset be the set itself:
$M_1\equiv \{1,2,\dots,m\}$. 
For $i\in \Ibb$ we denote the subspace $\Ecal_{M_i}\equiv \bigoplus_{j\in M_i}\Ecal_j$ where 
$\{\Ecal_1, \dots,\Ecal_m\}$ is a collection of mutually orthogonal subspaces of $\Ecal$.  
The complement to this space in $\Ecal$ is denoted $\Ecal_{M_i}^\circ\equiv \Ecal\setminus \Ecal_{M_i}$;  
likewise, denote the complement to the 
subset $M_i$ in $\{1,2,\dots,m\}$ by  
$M_i^\circ = \{1,2,\dots,m\}\setminus M_i$. 
The {\em affine} embedding of the subspace $\Ecal_{M_i}$ in $\Ecal$ at a 
point $z\in \Ecal$ is denoted $\Ecal_{M_i}\bigoplus \{z\}$;  the 
canonical embedding of $\Ecal_{M_i}$ in $\Ecal$ is thus $\Ecal_{M_i}\bigoplus \{0\}$ where it is 
understood that $0\in\Ecal$.  We use the corresponding notation for subsets $G\subset\Ecal$:  
$G_j\subset\Ecal_j$ and the affine embedding of a subset $G_{M_i}$ at a point 
$z\in G_{M_i^\circ}$ is given by $G_{M_i}\bigoplus\{z\}_{M_i^\circ}$.  
The blockwise mappings $\map{T_i}{\Ecal}{\Ecal}$ corresponding to this structure are defined by  
\begin{equation}\label{e:T blockwise}
[T_{i}(x)]_j\equiv 
                \begin{cases}
                    T'_j(x), & j\in M_{i},\\
                    x_j&\mbox{ else,}
                \end{cases}
\quad \fur\quad  
\map{T'_j}{\Ecal}{\Ecal_j}, \quad j=1,2,\dots,m. 
\end{equation}
Note that $T'_j$  is some action with respect to the $j$'th block 
in $\Ecal_j$, though with input from $x\in\Ecal$.

The measure space of indexes is denoted $(\Ibb, \Ical)$ and $\xi$ is an $\Ibb$-valued random variable
on a probability space. 
The random variables $\xi_k$ in the sequence
$(\xi_{k})_{k\in\mathbb{N}}$ (abbreviated $(\xi_{k})$) 
are independent and identically
distributed (i.i.d.) \ with $\xi_{k}$ distributed as
$\xi$ ($\xi_k\sim \xi$).   
At each iteration $k$ of the algorithm one selects at 
random a nonempty subset of blocks $M_{\xi_{k}}\subset\{1,2,\dots,m\}$ and performs an 
update to each block as follows:
\begin{algorithm}
\SetKwInOut{Output}{Initialization}
  \Output{Select a random variable $X_0$ with distribution $\mu$,  $t=(t_1, t_2, \dots,t_m)>0$, 
   and $(\xi_{k})_{k\in\Nbb}$ 
   an i.i.d.\ sequence with values on $\Ibb$ and  
   $X_0$ and  $(\xi_{k})$ independently distributed. Given $\map{T'_j}{\Ecal}{\Ecal_j}$
   for $j=1,2,\dots,m$.}
    \For{$k=0,1,2,\ldots$}{
            { 
            \begin{equation}\label{eq:spcd}
                X^{k+1}= T_{\xi_{k}}(X^{k})\where [T_{\xi_{k}}(X^k)]_j\equiv 
                \begin{cases}
                    T'_j(X^k), & j\in M_{\xi_{k}},\\
                    X_j^k&\mbox{ else}
                \end{cases}.
            \end{equation}
            }\\
    }
  \caption{Stochastic Block Iteration (SBI)}\label{algo:sbi}
\end{algorithm}

This is a special instance of a {\em random function iteration} studied in 
\cite{HerLukStu19a, HerLukStu22a, HerLukStu22b}.
Convergence of such an iteration is understood in the sense of distributions and is a consequence 
of two key properties:  that the mapping $T_i$ is almost 
{\em $\alpha$-firmly nonexpansive (abbreviated a$\alpha$-fne) 
in expectation} (\eqref{eq:paafne i.e.} and \eqref{eq:Phi aafneie}), 
and that the {\em invariant Markov transport discrepancy} defined in \eqref{eq:Psi} is 
{\em gauge metrically subregular} \eqref{e:metricregularity} at invariant measures.  
The latter of these two properties has been shown in many settings to be necessary for quantitative 
convergence of the iterates \cite{HerLukStu19a, LukTebTha18}.  
The first property, with the qualifier ``almost'' removed, 
is enough to guarantee that the sequence of 
measures is {\em asymptotically regular} with respect to the {\em Wasserstein metric}.  All this 
is formally defined below.

\subsection{Markov chains, measure-valued mappings, and stochastic fixed point problems}
\label{sec:SPMasMC}

The following assumptions hold throughout.
\begin{assumption}\label{ass:1}
  \begin{enumerate}[(a)]
  \item\label{item:ass1:indep} $\xi_{0},\xi_{1}, \ldots, \xi_{k}$
    are i.i.d random variables for all $k\in\Nbb$ on a probability space 
    with values on $\Ibb$.  The variable $X_0$ is an 
    random variable with values on $\Ecal$, independent from $\xi_k$. 
  \item\label{item:ass1:Phi} The function 
  $\map{\Phi}{\Ecal\times \Ibb}{\Ecal}$, $(x,i)\mapsto T_{i}x$ is measurable.
  \end{enumerate}
\end{assumption}

Let $(X_{k})_{k \in \mathbb{N}}$ be a sequence of random variables with values on $G\subset \Ecal$. 
  Recall that a Markov chain  with {\em transition kernel} $p$ satisfies 
  \begin{enumerate}[(i)]
  \item $\cpr{X_{k+1} \in A}{X_{0}, X_{1}, \ldots, X_{k}} =
    \cpr{X_{k+1} \in A}{X_{k}}$;
  \item $\cpr{X_{k+1} \in A}{X_{k}} = p(X_{k},A)$
  \end{enumerate}
for all $k \in
  \mathbb{N}$ and $A \in \mathcal{B}(G)$ almost surely in probability,  
  $\mathbb{P}$-a.s.  In \cite{HerLukStu22a} it is shown that the sequence 
  of random variables $(X_k)$ generated 
by Algorithm \ref{algo:sbi} is a Markov chain with transition kernel  $p$ given by  
\begin{equation}\label{eq:trans kernel}
  (x\in G) (A\in
  \mathcal{B}(G)) \qquad p(x,A) \equiv 
  \mathbb{P}(T_{\xi}x \in A)
\end{equation}
for the measurable \emph{update function} $\map{\Phi}{G\times  \Ibb}{G}$ given by 
$\Phi(x,i)\equiv T_{i}x$.  

The Markov operator $\mathcal{P}$ associated with this Markov chain 
is defined pointwise for a measurable 
function 
$\map{f}{G}{\mathbb{R}}$ via
\begin{align*}
  (x\in G)\qquad \mathcal{P}f(x):= \int_{G} f(y) p(x,\dd{y}),
\end{align*}
when the integral exists. Note that
\begin{align*}
  \mathcal{P}f(x) =  \int_{\Omega}
  f(T_{\xi(\omega)}x) \mathbb{P}(\dd{\omega})= \int_{\Ibb}
  f(T_{i}x) \mathbb{P}^{\xi}(\dd{i}).
\end{align*}

Let $\mu\in \mathscr{P}(G)$.  The
dual Markov operator acting on a measure $\mu$ is indicated by action on 
the right by $\mathcal{P}$:
\begin{align*}
  (A \in \mathcal{B}(G))\qquad (\mathcal{P}^{*}\mu) (A):=
  (\mu\mathcal{P}) (A) := \int_{G} p(x,A) \mu(\dd{x}).
\end{align*}
The distribution of the $k$'th iterate of the Markov chain generated
by Algorithm \ref{algo:sbi} is therefore easily represented as follows: $\mathcal{L}(X_{k}) =
\mu_0 \mathcal{P}^{k}$, where $\mathcal{L}(X)$ denotes the law of the random
variable $X$.  Of course in general random variables do not converge, but distributions associated with 
the sequence of random variables $(X_k)$ of Algorithm \ref{algo:sbi}, if they converge to 
anything, do so to {\em invariant measures} of the associated Markov operator. 
  An invariant measure of the Markov operator $\mathcal{P}$ is any distribution $\pi\in \mathscr{P}$
  that satisfies $\pi \mathcal{P} = \pi$.  The
set of all invariant probability measures is denoted by 
$\inv \mathcal{P}$.  The underlying problem we seek to solve is to 
\begin{align}
  \label{eq:stoch_fix_probl}
  \mbox{Find}\qquad \pi\in\inv\mathcal{P}.
\end{align}
This is the {\em stochastic fixed
  point problem} studied in \cite{HerLukStu22a, HerLukStu22b}.  
  When the mappings $T_i$ have common fixed points, the problem reduces to the 
  {\em stochastic feasibility} problem studied in \cite{HerLukStu19a}.  

Let $(\nu_{k})$ be a sequence of probability measures on $G\subset\Ecal$, and let  
$C_{b}(G)$ denote the set of bounded and continuous
functions from $G$ to $\mathbb{R}$.  The sequence 
$(\nu_{k})$ is said to converge in distribution to $\nu$ whenever $\nu \in
\mathscr{P}(G)$ and for all $f \in C_{b}(G)$ it holds that $\nu_{k} f
\to \nu f$ as $k \to \infty$, where $\nu f := \int f(x) \nu(\dd{x})$. 
In other words,  a sequence of random variables $(X_{k})$ converges in 
distribution if their laws $(\mathcal{L}(X_{k}))$ do.  We use the {\em weighted Wasserstein metric}
for the space of measures. Let 
  \begin{equation}\label{eq:p-probabiliy measures}
       \mathscr{P}_{2}(G) = \set{\mu \in \mathscr{P}(G)}{ \exists\, x
      \in G \,:\, \int \|x-y\|_\pbf^2 \mu(\dd{y}) < \infty}
  \end{equation}
  where $\|\cdot\|_\pbf$ is the Euclidean norm weighted by $\pbf$.  
  This will be made explicit below.  
  The Wasserstein $2$-metric on $\mathscr{P}_{2}(G)$, with respect to 
  the weighted Euclidean norm $\|\cdot\|_\pbf$ denoted $d_{W_{2,\pbf}}$, is defined by 
 \begin{equation}\label{eq:Wasserstein}
  d_{W_{2,\pbf}}(\mu, \nu)\equiv \paren{\inf_{\gamma\in \Ccal(\mu, \nu)}\int_{G\times G} 
\|x-y\|_\pbf^2\gamma(dx, dy)}^{1/2}
 \end{equation}
 where $\Ccal(\mu, \nu)$ is the set of couplings of $\mu$ and $\nu$:
   \begin{align}
    \label{eq:couplingsDef}
    \Ccal(\mu,\nu) := \set{\gamma \in \mathscr{P}(G\times G)}{ \gamma(A
      \times G) = \mu(A), \, \gamma(G\times A) = \nu(A) \quad \forall A
      \in \mathcal{B}(G)}.
  \end{align}

The principle mode of convergence in distribution that we use is convergence in distribution 
of the sequence 
  $(\mathcal{L}(X_{k}))$ to a probability measure $\pi \in \mathscr{P}(G)$, i.e.\
  for any $f \in C_{b}(G)$
  \begin{align*}
    \mathcal{L}(X_{k})f = \mathbb{E}[f(X_{k})] \to \pi f, 
    \qquad \text{as } k \to \infty.
  \end{align*}
This is a stronger form of convergence than convergence of Ces\`aro averages sometimes seen
in the literature. Since we are considering the The Wasserstein $2$-metric, 
convergence in this metric implies that also the second moments converge in this metric.  
For more background on the analysis of sequences of measures we refer 
interested readers to 
\cite{Billingsley, stroock2010probability, Villani2008, Szarek2006, Hairer2021}.   

\subsection{Stochastic blockwise splitting algorithms}
The concrete targets of the analysis presented here are 
two fundamental templates for solving problems of the form \eqref{e:P1}, 
forward-backward splitting as formulated in \cite{salzo2021} and 
Douglas-Rachford splitting; the latter has not been studied in this context. 

Denote by $\mmap{\partial_{x_j} f}{\Ecal}{\Ecal_j}$ the partial
limiting subdifferential of $f$ with respect to the block $x_j\in \Ecal_j$:
\begin{equation}\label{eq:psd}
 \partial_{x_j} f(\xbar)\equiv\set{v\in\Ecal_j}{ f(x)\geq f(\xbar)+
 \ip{v\bigoplus\{0\}}{x-\xbar} + o\{\|x-\xbar\|\}, 
 x\in \Ecal_j\bigoplus\{\xbar\}}.
\end{equation}
When $f$ is continuously differentiable, then this coincides with the partial gradient 
$\map{\nabla_{x_j}f}{\Ecal}{\Ecal_j}$.  
The prox mapping of a function 
$\mmap{h}{\Ecal}{\extre}$ is defined by
\begin{align}\label{e:prox}
\prox_{h, \lambda} (x) \coloneqq \argmin_{y \in \Ecal} \left \{ h(y)+ \frac{1}{2 \lambda} \norm{y - x}^2\right \}. 
\end{align}
The prox mapping is nonempty and single-valued whenever $h$ is proper, lsc and convex \cite{Moreau65}. To 
allow for generalization to {\em nonconvex} functions we use instead the {\em resolvent} 
$\mmap{J_{\partial h, \lambda}}{\Ecal}{\Ecal}$:
\begin{align}\label{e:resolvent}
J_{\partial h, \lambda} (x) \coloneqq \set{y \in \Ecal}{\paren{\lambda\sd h + \Id}(y) \ni x}.
\end{align}
It is clear from this that, in general, $\prox_{h,\lambda}(x)\subset J_{\partial h, \lambda}(x)$ for all $x$. 

Note that $g_j$ defined in \eqref{e:gi} is just the extension by zero of $h_j$ to a mapping on $\Ecal$.  This yields
\begin{equation}
 J_{\partial g_j, \lambda_j}(x) = \paren{\begin{array}{c}
                              x_1\\x_2\\\vdots\\x_{j-1}\\J_{\partial h_j, \lambda_j}(x_j)\\
                              x_{j+1}\\\vdots\\x_m
                             \end{array}}
                             \quad\mbox{ and }\quad
 \paren{J_{\partial g_j, \lambda_j} - \Id}(x) = 
 \paren{J_{\partial h_j, \lambda_j}(x_j) - x_j}\bigoplus\{0\}.
\end{equation}
Let $\map{\sd_j f}{\Ecal}{\Ecal}$ denote the 
canonical embedding of $\sd_{x_j} f$ by zero into $\Ecal$: 
\begin{equation}\label{e:nabla_j}
 \sd_j f(x) \equiv \sd_{x_j}f(x)\bigoplus\{0\}.
\end{equation}
The corresponding resolvent,
$J_{\sd_j f, \lambda}(x)$  is given by
\begin{equation}
 J_{\sd_j f, \lambda}(x) \equiv \paren{\begin{array}{c}
                              x_1\\x_2\\\vdots\\x_{j-1}\\J_{\sd f_j(\cdot; x), \lambda}(x_j)\\
                              x_{j+1}\\\vdots\\x_m
                             \end{array}}
\end{equation}
where 
$\map{f_j(\cdot; x)}{\Ecal_j}{\Rbb}$, with $x\in\Ecal$ a {\em parameter}, denotes
\begin{subequations}\label{e:f_j}
\begin{equation}
f_j(y;x) \equiv  
  f\paren{x+(y-x_j)\bigoplus \{0\}}
\end{equation}
so that   
\begin{eqnarray}
\sd f_j(y;x) &=& \sd_{x_j} f\paren{x+(y-x_j)\bigoplus \{0\}}\und\label{e:nabla f_j}\\ 
J_{\sd f_j(\cdot; x), \lambda}(x_j)&=& \set{y\in\Ecal_j}{y+\sd_{x_j} f(x+(y-x_j)\bigoplus\{0\}) \ni x_j}.
\label{e:J nabla f_j}
\end{eqnarray}
\end{subequations}
We recognize that the resolvent of a function that is not fully separable is not considered
{\em prox friendly} from a computational standpoint, and this can only be evaluated 
numerically with some error.  The framework presented here is well suited for algorithms 
with numerical error, and this is discussed at some length in 
\cite[Section 4]{HerLukStu22a}.  In the interest of keeping the presentation simple, we 
present results for exact evaluation of all the relevant operators;  the incorporation of 
appropriate noise models for inexact computation builds on the structure introduced here and 
does not require any assumption of summable errors or increasing accuracy, 
though the noise model does require some careful consideration
(see \cite[Section 4.3]{HerLukStu22a}).  

The abstract template is Algorithm \ref{algo:sbi} where the mappings $T'_j$ 
specialize to 
\[ T'_j(x) \equiv
                \frac{1}{2}\paren{R_{\sd f_j(\cdot;y), t_j}R_{\partial h_j, t_j}(x_j) + x_j}, \quad
                y = R_{\partial g_j, t_j}(x)
                \qquad (\mbox{blockwise Douglas-Rachford})
\]
where 
\[ R_{\partial h_j, t_j}(x_j) = 2 J_{\partial h_j, t_j}(x_j)-x_j\und  R_{\sd f_j(\cdot;x), t_j}(x_j) 
\equiv 2J_{\sd f_j(\cdot;x)}(x_j)-x_j. 
\]
or, when $f$ is continuously differentiable,
\[
T'_j(x) \equiv
                J_{\partial h_j, t_j}\paren{x_j - t_j\nabla_{x_j} f(x)} \qquad (\mbox{blockwise forward-backward}).
\]
Using the resolvent instead of the prox mapping, the blockwise forward-backward algorithm studied in 
\cite{salzo2021} consists of iterations of randomly selected mappings $\map{T^{FB}_i}{\Ecal}{\Ecal}$:
\begin{equation}\label{e:TFB_i}
 T^{FB}_{i}\equiv 
                \paren{\Id + \sum_{j\in M_{i}}\paren{J_{\partial g_{j}, t_j}\paren{\Id-{t_j}\nabla_j f} - \Id}}
                \quad (i\in \Ibb).
\end{equation}
\begin{algorithm}
\SetKwInOut{Output}{Initialization}
  \Output{Select a random variable $X_0$ with distribution $\mu$,  $t=(t_1, t_2, \dots,t_m)>0$, 
   and $(\xi_{k})_{k\in\Nbb}$ 
   an i.i.d.\ sequence with values on $\Ibb$ and  
   $X_0$ and  $(\xi_{k})$ independently distributed. }
    \For{$k=0,1,2,\ldots$}{
            {
            \begin{subequations}\label{eq:sbfb}
            \begin{equation}\label{eq:sbfb_0}
                X^{k+1}= T^{FB}_{\xi_{k}}(X^k)\equiv 
                \paren{\Id + \sum_{j\in M_{\xi_{k}}}\paren{J_{\partial g_{j}, t_j}\paren{\Id-{t_j}\nabla_j f} - \Id}}(X^{k}),
            \end{equation}
            or equivalently\\
            \For{$j=0,1,2,\ldots,m$}{
            \begin{equation}\label{eq:sbfb_i}
                X_j^{k+1}= [T^{FB}_{\xi_{k}}(X^k)]_j\equiv 
                \begin{cases}
                    J_{\partial h_{j}, t_j}\paren{X^k_j-{t_j}\nabla_{x_j} f(X^k)}
                      & \mbox{ if } j\in M_{\xi_{k}}\\
                    X^k_j&\mbox{ else }.
                \end{cases}
            \end{equation}}
            \end{subequations}
            }
    }
  \caption{Stochastic Blockwise Forward-Backward Splitting (S-BFBS)}\label{algo:sblfb}
\end{algorithm}

The blockwise Douglas-Rachford algorithm 
consists of iterations of randomly selected mappings $\map{T^{DR}_i}{\Ecal}{\Ecal}$:
\begin{equation}\label{e:TDR_i}
 T^{DR}_{i}\equiv 
                \frac{1}{2}\paren{\sum_{j\in M_{i}}\paren{R_{\sd_j f, t_j} 
                R_{\partial g_j,t_j} - \Id} + 2\Id}
                \quad (i\in \Ibb).
\end{equation}
\begin{algorithm}
\SetKwInOut{Output}{Initialization}
  \Output{Select a random variable $X_0$ with distribution $\mu$,  $t=(t_1, t_2, \dots,t_m)>0$, 
   and $(\xi_{k})_{k\in\Nbb}$ 
   an i.i.d.\ sequence with values on $\Ibb$ and  
   $X_0$ and  $(\xi_{k})$ independently distributed. }
    \For{$k=0,1,2,\ldots$}{
            { 
            \begin{subequations}\label{eq:sbdr}
            \begin{equation}\label{eq:sbdr0}
                X^{k+1}= T^{DR}_{\xi_{k}}X^k\equiv 
                \frac{1}{2}\paren{\sum_{j\in M_{\xi_{k}}}\paren{R_{\sd_j f,t_j}R_{\partial g_{j}, t_j} 
                 - \Id} + 2\Id}(X^{k})
            \end{equation}
            or equivalently\\
            \For{$j=0,1,2,\ldots,m$}{
            \begin{equation}\label{eq:sbdr_i}
                X_j^{k+1}= [T^{DR}_{\xi_{k}}(X^k)]_j\equiv 
                \begin{cases}
                    \frac{1}{2}\paren{R_{\sd_j f,t_j}R_{\partial h_{j}, t_j}(X^k_j) + X^k_j}
                      & \mbox{ if } j\in M_{\xi_{k}}\\
                    X^k_j&\mbox{ else }.
                \end{cases}
            \end{equation}}
            \end{subequations}
            }
    }
  \caption{Stochastic Blockwise Douglas-Rachford Splitting (S-BDRS)}\label{algo:sbdr}
\end{algorithm}
In addition to its own merits, in the convex setting 
the Douglas-Rachford algorithm has the interpretation 
as the ADMM algorithm \cite{Glowinski75} applied to the ``pre-primal'' problem to \eqref{e:P1}
\cite{Gabay83, Eckstein}: 
\begin{equation}\label{eq:pre primal}
 \ucmin{p(x)+q(Ax)}{x\in\Rn}\quad  \where p^*(-A^Tx) = f(x)\und g(x) = q^*(x).
\end{equation}
The stochastic blockwise Douglas-Rachford Algorithm \ref{algo:sbdr} therefore can be understood 
as a stochastic blockwise ADMM algorithm for solving \eqref{eq:pre primal}.  The discussion above 
about the separability of $f$ is yet another way of understanding the observed computational 
difficulty of implementing this algorithm; it is quite unlikely that $f$ given by 
\eqref{eq:pre primal} will be separable in the standard basis and therefore the 
resolvent \eqref{e:J nabla f_j} will have to be computed numerically.  
Alternative primal-dual methods that circumvent this are the topic of future research.  

Before we begin, however, it will be helpful to give an example delineating consistent from 
inconsistent feasibility. 

\begin{eg}[consistent/inconsistent stochastic feasibility problems]
Examples for partially separable optimization and blockwise algorithms abound, particularly 
in machine learning, but seldom is the distinction made between consistent and inconsistent 
problems.  This is illustrated here for the problem of set feasibility, or, when
feasible points don't exist, best approximation.  Consider the problem 
\[
 \Find \xbar\in \cap_{j=1}^m\Omega
\]
where $\Omega_j\subset\Rn$ are closed sets.  This can be recast as the following optimization problem 
on the product space $(\Rn)^m$:
\begin{equation}\label{e:P1 feas}
 \mbox{minimize } f(x)+\sum_{j=1}^m \iota_{\Omega_j}(x_j)
\end{equation}
where $x_j\in\Rn$, 
\[
\iota_\Omega(x) = \begin{cases} 0&\mbox{ when }x\in\Omega\\ +\infty&\mbox{ else,}\end{cases} 
\]
and $f$ is some reasonable coupling function that promotes similarity between the blocks $x_j$.  
In the context of problem \eqref{e:P1} 
$h_j = \iota_{\Omega_j}$.  Common instances of the coupling function  
are $f(x) =\tfrac12 d(x, D)^2$ for 
$D\equiv\set{x=(x_1,x_2, \dots, x_m)}{x_i=x_j~\forall i\neq j}$ or the more strict 
indicator function $f(x)=\iota_D(x)$.  The prox operators associated with the indicator 
functions are just projectors, while the gradient of the function $f$ in the smooth case
can be constructed from the projection onto $D$ (just the averaging operator).

When $\cap_{j=1}^m\Omega\neq\emptyset$, the solutions 
to the feasibility problem and problem \eqref{e:P1 feas} coincide for both instances of $f$. 
In this case the blockwise operators $T_i$ in \eqref{algo:sblfb} and \eqref{algo:sbdr} 
have common fixed points, which are (perhaps not exclusively) points where the sets intersect,  
and so, when all goes well, fixed points of these algorithms coincide with points in $\cap_{j=1}^m\Omega$; 
at the very least fixed points of the algorithms coincide with {\em critical points}.  
Viewed as random function iterations, 
the iterates of \eqref{algo:sblfb} and \eqref{algo:sbdr} in this consistent case are random variables 
whose distributions converge to delta functions with support in $\cap_{j=1}^m\Omega$ and the algorithms 
converge to solutions of a {\em stochastic feasibility problem} studied in \cite{HerLukStu19a}:
\[
 \Find \xbar\in \set{x}{\mathbb{P}(x\in\Fix T_\xi)=1}.
\]

If the intersection is empty, as will often be the case in practice regardless of noise considerations, 
then it is easy to see that the blockwise operators in \eqref{algo:sblfb} and \eqref{algo:sbdr} 
do not have common fixed points when $f(x)=\iota_D(x)$.  The random algorithms do not 
have fixed points in this case, but viewed as random function iterations, the distributions of the 
iterates converge to invariant measures of the Markov operator corresponding to either Algorithm 
\eqref{algo:sblfb} or \eqref{algo:sbdr}.  These algorithms therefore find solutions to the more general 
stochastic fixed point problem \ref{eq:stoch_fix_probl} studied in \cite{HerLukStu22a}.
How to interpret such invariant measures is an open issue in general.  For this example,
in the case of just two convex sets with
empty intersection, the invariant probability measures will consist of 
equally weighted pairs of delta functions centered at best approximation pairs between the sets.

The numerical 
behavior of deterministic 
versions of \eqref{algo:sblfb} and \eqref{algo:sbdr}, and many others, has been 
thoroughly studied for the broad class of {\em cone and sphere problems}, which includes
sensor localization, phase retrieval, and computed tomography \cite{LukSabTeb18}.  
For the example of set feasibility presented here, convergence depends on the regularity
properties of the projectors onto the respective sets, which as shown in \cite{LukNguTam18}
is derived from the regularity of the sets.  The main contribution of this article is to 
show that randomization can lead to Markov operators with better regularity than that of the 
individual operators generating its transition kernel.  
\end{eg}

\section{Regularity}\label{s:regularity}
Our main results concern convergence of Markov chains under 
regularity assumptions that are lifted from the generating mappings $T_i$.  In 
\cite{LukNguTam18} a framework was developed for a quantitative 
convergence analysis of set-valued mappings $T_{i}$ that are one-sided Lipschitz 
continuous in the sense of set-valued-mappings 
with Lipschitz constant slightly greater than 1.  We begin with the regularity of 
$T_{i}$ and follow this through to the regularity of the resulting Markov operator.

\subsection{Almost $\alpha$-firmly nonexpansive mappings}
\label{s:aafneie}
Let $G\subset \Ecal$ and let 
$\mmap{F}{G}{\Ecal}$.
 The mapping $F$  is said to be \emph{pointwise  
  almost nonexpansive at $x_0\in G$ on $G$}
whenever  
\begin{equation}\label{eq:pane}
\exists \epsilon\in[0,1):\quad  \|x^+ - x_0^+\| \le \sqrt{1+\epsilon}\,\|x - x_0\|, 
\qquad \forall x \in G, \forall x^+\in Fx, x_0^+\in Fx_0.
\end{equation}
The {\em violation} is a value of $\epsilon$ for which \eqref{eq:pane} holds.
    When the above inequality holds for all $x_0\in G$ then $F$ is said to be 
    {\em almost nonexpansive on $G$}.  When $\epsilon=0$ the 
    mapping $F$ is said to be 
    {\em (pointwise) nonexpansive}.  
 The mapping $F$ is said to be {\em pointwise almost $\alpha$-firmly 
nonexpansive  at $x_0\in G$ on $G$}, abbreviated {\em pointwise a$\alpha$-fne} 
whenever
\begin{eqnarray}
&&
	\exists \epsilon\in[0,1)\mbox{ and }\alpha\in (0,1):\nonumber\\
	  \label{eq:paafne}&&	\quad	\|x^+ - x_0^+\|^2 \le (1+\epsilon) \|x - x_0\|^2 - 
      \tfrac{1-\alpha}{\alpha}\psi(x,x_0, x^+, x_0^+)\\
&&      \qquad\qquad\qquad\qquad\qquad\quad \forall x \in G, \forall x^+\in Fx, \forall x_0^+\in Fx_0,
		\nonumber
\end{eqnarray}
where the {\em transport discrepancy} $\psi$ of  $F$ at $x, x_0$, $x^+\in Fx$ and $x_0^+\in Fx_0$
is defined by  
\begin{eqnarray}
&&\!\!\!\!\!\!\!\!\psi(x,x_0, x^+, x_0^+)\equiv \nonumber\\
\label{eq:psi}
&&\!\!\!\!\!\!\!\! \|x^+- x\|^2+ \|x_0^+ - x_0\|^2 + \|x^+ - x_0^+\|^2 + 
\|x - x_0\|^2  - \|x^+ - x_0\|^2   - \|x - x_0^+\|^2.
\end{eqnarray}
When the above inequality holds for all $x_0\in G$ then $F$ is said to be 
{\em a$\alpha$-fne on $G$}. 
The {\em violation} is the constant  
$\epsilon$ for which \eqref{eq:paafne} holds.
When $\epsilon=0$ the mapping $F$ is said to be 
    {\em (pointwise) $\alpha$-firmly nonexpansive}, abbreviated {\em (pointwise) $\alpha$-fne}.  

The transport discrepancy $\psi$ is a central object for characterizing the regularity of mappings 
in metric spaces and ties the regularity of the mapping to the geometry of the space.
A short calculation shows that, in a Euclidean space, this has 
the representation 
\begin{equation} \label{eq:nice ineq}
\psi(x,x_0, x^+, x_0^+) =  \|(x-x^+)-(x_0-x_0^+)\|^2.
\end{equation}

The definition of pointwise a$\alpha$-fne
mappings in Euclidean spaces appeared first in \cite{LukNguTam18}.
This generalizes the notion of 
{\em averaged} mappings dating back to Mann, Krasnoselskii, and others
\cite{mann1953mean, krasnoselski1955, edelstein1966, BruckReich77,BaiBruRei78}.

A partial blockwise mapping $T_i$ that is $\alpha$-fne on an affine subspace
$\Ecal_{M_i}\bigoplus \{z\}$ may not be $\alpha$-fne on $\Ecal$, as the 
next example from \cite[Remark 3.9]{Kartamyschew} shows. 
\begin{eg}\label{eg:T_i not afne} 
Let $\Ecal = \Rtw$ and define $f(x_1, x_2) = (x_1+x_2)^2$, $g_1(x_1,x_2)=h_1(x_1) = 0$
and $g_2(x_1, x_2) = h_2(x_2)= x_2^2$.  Here $f$ is convex and differentiable with global 
gradient Lipschitz constant $L=4$ and the functions $g_j$ are clearly convex. The proximal 
gradient algorithm applied to the function $F = f +\sum_{j=1}^2 g_j$ is 
$x^{k+1}=T^{FB}(x^k)=\prox_g(\Id- t\nabla f)(x^k)$.
For all $t\in(0,1/2)$ it can be shown that the fixed point mapping $T^{FB}$ is 
$\alpha$-firmly nonexpansive with the unique fixed point $(0,0)$, the global minimum 
of the objective function $F$.  Hence 
from any initial point $x^0$ this iteration converges to the global minimum $(0,0)$.  
A blockwise implementation of this algorithm would involve
computing the proximal gradient step with respect to $x_1$, leaving $x_2$ fixed;  that is 
at some iterations $k$ one computes
\begin{equation}\label{eq:ex TFB_1}
 x^{k+1} = T^{FB}_1(x^k)\equiv \prox_{g_1}((\Id - t\nabla_{x^k_1}f)(x^k_1, x^k_2) = ((1-2t)x^k_1-2tx^k_2, x^k_2).
\end{equation}  
A straightforward calculation shows that the blockwise mapping $T^{FB}_1$
is not $\alpha$-fne on $\Rtw$ for any $t>0$, although it is $\alpha$-fne on 
$\Rbb\times \{z\}$ for any $z\in\Rbb$ whenever $t\in (0,1/2)$.  
Being $\alpha$-fne on $\Rbb\times \{z\}$ for 
any $z\in\Rbb$ is not much help, however, since this means that repeated
application of $T^{FB}_1$ defined by \eqref{eq:ex TFB_1} converges 
to the  minimum of $F$ restricted to the affine subspace $\Rbb\times \{z\}$, 
namely $(-z, z)$. 
\end{eg}
In light of the above counterexample, Theorem \ref{t:afneie} below 
shows how randomization in the blockwise forward-backward algorithm 
restores the $\alpha$-fne property {\em in expectation} \cite[Definition 3.6]{HerLukStu22b}.  
This is the fixed point analog to descents in expectation introduced in \cite{salzo2021}.

In the stochastic setting 
we consider only {\em single-valued} mappings $T_i$ 
that are a$\alpha$-fne in expectation.
We can therefore write $x^+ = T_i x$ instead
of always taking some selection $x^+\in T_i x$ (which then raises issues of 
measurability and so forth).  
On a closed subset $G\subset\Ecal$ for a general self-mapping $\map{T_i}{G}{G}$ 
for $i\in \Ibb$, the mapping  
   $\map{\Phi}{G\times \Ibb}{G}$ be given by $\Phi(x,i) = T_{i}x$  
   is 
said to be \emph{pointwise almost nonexpansive in expectation at $x_0\in G$} on 
$G$, abbreviated {\em pointwise almost nonexpansive in expectation}, whenever 
    \begin{align}\label{eq:panee}
		\exists \epsilon\in [0,1):\quad 
		\mathbb{E}\ecklam{\|\Phi(x,\xi)- \Phi(x_{0},\xi)\|} \le \sqrt{1+\epsilon}\,
		\|x-x_0\|, \qquad \forall x \in G.
    \end{align}
    When the above inequality holds for all $x_0\in  G$ then
    $\Phi$ is said to be {\em almost nonexpansive in expectation on
      $ G$}.  As before, the violation is a value of  $\epsilon$ for which 
  \eqref{eq:panee} holds. When the violation is $0$, the qualifier ``almost'' is dropped. 
The mapping $\Phi$ is said to be {\em pointwise almost $\alpha$-firmly 
nonexpansive  in expectation at $x_0\in  G$} on $ G$, abbreviated 
{\em pointwise a$\alpha$-fne in expectation}, 
  whenever
  \begin{eqnarray}\label{eq:paafne i.e.}
&&\exists \epsilon\in [0,1), \alpha\in (0,1):\quad \forall x \in  G,\\
&&\quad\mathbb{E}\ecklam{\|\Phi(x,\xi)-\Phi(x_0,\xi)\|^2}\leq 
(1+\epsilon)\|x-x_0\|^2 - 
\tfrac{1-\alpha}{\alpha}\mathbb{E}\left[\psi(x,x_0, \Phi(x,\xi), \Phi(x_0,\xi))\right].
\nonumber
\end{eqnarray}
When the above inequality holds for all $x_0\in  G$ then $\Phi$ is said to be 
{\em almost $\alpha$-firmly nonexpansive (a$\alpha$-fne) in expectation on $ G$}.  The 
violation is a value of $\epsilon$ for which \eqref{eq:paafne i.e.} holds. 
When the violation is $0$, the qualifier ``almost'' is dropped and the abbreviation 
{\em $\alpha$-fne in expectation} is used. The defining inequalities 
\eqref{eq:panee} and  \eqref{eq:paafne i.e.}
will be amended below in \eqref{eq:Phi aafneie} to account for {\em weighted norms}.  

The next result, derived from \cite[Proposition 5.5]{Kartamyschew} 
shows in particular that any collection of self-mappings $\{T_i\}_{i\in \Ibb}$ on $ G\subset\Ecal$ 
that is a$\alpha$-fne on $G_{M_i}\bigoplus\{z\}$ is a$\alpha$-fne in 
expectation with respect to a weighted norm on $G$. In particular, denote
by $\eta_i$ the probability of selecting the $i$'th collection of blocks, $M_i$, and 
let $p_j$ denote the probability that the 
$j$'th block is among the randomly selected collection of blocks:  
\begin{equation} 0<p_j = \sum_{i\in\Ibb}\eta_i\cdot\chi_{M_i}(j)\leq 1
 \where 
    \chi_{M_i}(j) = \begin{cases}
        1&\mbox{ if }j\in M_i\\
        0& \mbox{else}
     \end{cases}
    \quad     (j=1,2,\dots,m).
     \label{e:pj}
\end{equation}
Define the corresponding weighted norm
\begin{equation}\label{e:norm pbf}
\norm{z}_\pbf\equiv \paren{\sum_{j=1}^m \tfrac{1}{p_j}\|z_j\|_{\Ecal_j}^2}^{1/2}.
\end{equation}

\begin{thm}[almost $\alpha$-firmly nonexpansive in expectation 
(a$\alpha$-fne in expectation)]\label{t:afneie}
 Let the single-valued self-mappings $\{T_i\}_{i\in \Ibb}$ on the 
 subset $G\subset\Ecal$ 
satisfy
\begin{enumerate}[(a)]
 \item\label{t:afneie a} for each $i$, $T_i$ is the identity mapping on $\Ecal_{M_i^{\circ}}$;
 \item\label{t:afneie b} $T_1$ is a$\alpha$-fne on $ G$ with constant $\alphabar$ and 
violation no greater than $\epsilonbar$ where 
$M_1\equiv \{1,2,\dots,m\}$.  
\end{enumerate}
Then 
\begin{enumerate}[(i)]
 \item\label{t:afneie i} for all $i$ and each $z\in G$,  
 $T_i$ is a$\alpha$-fne on $G_{M_i}\bigoplus\{z\}_{M_i^\circ}$ 
with constant at most $\alphabar$ and 
violation no greater than $\epsilonbar$;  
 \item\label{t:afneie ii} the mapping $\map{\Phi}{ G\times \Ibb}{ G}$ given by 
$\Phi(x,i) = T_{i}x$ satisfies 
\begin{subequations}\label{eq:Phi summary}
\begin{equation}\label{eq:Phi aafneie}
 \Ebb\ecklam{\norm{\Phi(x, \xi)- \Phi(y,\xi)}_\pbf^2}\leq 
 (1+\pbar\epsilonbar)\|x-y\|_\pbf^2 - \tfrac{1-\alphabar}{\alphabar}
  \Ebb\ecklam{\psi_\pbf(x,y,\Phi(x,\xi),\Phi(y,\xi))}\quad\forall x,y\in G
\end{equation}
where
\begin{equation}
\psi_\pbf(x,y, \Phi(x,i),\Phi(y,i))\equiv 
\|\paren{x-\Phi(x,i)} - \paren{y-\Phi(y,i)}\|_\pbf^2
\und \pbar\equiv\max_j\{p_j\}.\label{eq:psi_p}
\end{equation}
\end{subequations}
\end{enumerate}
\end{thm}
A mapping $\map{\Phi}{ G\times \Ibb}{ G}$ that satisfies 
\eqref{eq:Phi aafneie} is called 
{\em a$\alpha$-fne in expectation}  
with respect to the weighted norm $\|\cdot\|_\pbf$ with constant $\alphabar$ 
and violation no greater than $\pbar\epsilonbar$.  

\begin{proof}
The proof of part \eqref{t:afneie i} follows immediately from the observation that $T_i$
on $ G_{M_i}\bigoplus\{z\}_{M_i^{\circ}}$ 
is equivalent to $T_1$ restricted to the same subset.

To see part \eqref{t:afneie ii}, fix any $x,y\in G$,  
and let $\map{T'_j}{ G_j}{ G_j}$ ($j=1,2,\dots,m$)
be the $j$'th block mapping for $j\in M_i$.  
Hence,  
$T_i(x) = P_{\Ecal_{M_i}}T_i(x) + P_{\Ecal_{M_i^\circ}}(x)$ and 
$T_1(x)=\bigoplus_{j=1}^m T'_j(x)$ where $\map{P_{\Ecal_{M_i}}}{\Ecal}{\Ecal}$
is the orthogonal projection onto the subspace $\Ecal_{M_i}$ and likewise for 
$P_{\Ecal_{M_i^\circ}}$.  
We begin with the left hand side of the defining inequality:
\begin{eqnarray}
 \Ebb\ecklam{\norm{\Phi(x, \xi)- \Phi(y,\xi)}_\pbf^2}
    &=& \sum_{i\in\Ibb}\eta_i\norm{T_i(x)- T_i(y)}_\pbf^2\nonumber\\
    &=& \sum_{i\in \Ibb}\eta_i\norm{\paren{P_{\Ecal_{M_i}}T_i(x)+ P_{\Ecal_{M_i^\circ}}(x)}- 
        \paren{P_{\Ecal_{M_i}}T_i(y) + P_{\Ecal_{M_i^\circ}}(y)}}_\pbf^2\nonumber\\
    &=& \sum_{i\in \Ibb}\eta_i
        \paren{\norm{P_{\Ecal_{M_i}}T_i(x) - P_{\Ecal_{M_i}}T_i(y)}_\pbf^2+ 
        \norm{P_{\Ecal_{M_i^\circ}}(x-y)}_\pbf^2}\nonumber\\
    &=& \sum_{i\in \Ibb}\eta_i\paren{\sum_{j\in M_i}\tfrac{1}{p_j}\norm{T'_{j}(x) - T'_{j}(y)}_{\Ecal_j}^2+ 
        \sum_{k\in M_i^\circ}\tfrac{1}{p_k}\norm{x_k-y_k}_{\Ecal_k}^2}.\label{e:the globe}
\end{eqnarray}
Then \eqref{e:the globe} rearranges to 
\begin{eqnarray}
 \Ebb\ecklam{\norm{\Phi(x, \xi)- \Phi(y,\xi)}_\pbf^2}
    &=& \sum_{i\in \Ibb}\eta_i\paren{\sum_{j\in M_i}\tfrac{1}{p_j}\norm{T'_{j}(x) - T'_{j}(y)}_{\Ecal_j}^2+ 
        \sum_{k\in M_i^\circ}\tfrac{1}{p_k}\norm{x_k-y_k}_{\Ecal_k}^2}\nonumber\\
    &=& \sum_{j=1}^m p_j\tfrac{1}{p_j}\norm{T'_{j}(x) - T'_{j}(y)}_{\Ecal_j}^2+ 
        (1-p_j)\tfrac{1}{p_j}\norm{x_j-y_j}_{\Ecal_j}^2\nonumber\\
    &=& \norm{T_1(x)-T_1(y)}^2 - \norm{x-y}^2 + \norm{x-y}_\pbf^2.\label{eq:last}
\end{eqnarray}
We simplify the expectation of the weighted transport discrepancy \eqref{eq:psi_p} next. 
\begin{eqnarray}
 \Ebb\ecklam{\psi_\pbf(x,y,\Phi(x,\xi),\Phi(y,\xi))} &=& 
        \sum_{i\in \Ibb}\eta_i
        \paren{\norm{\paren{x-T_i(x)} - \paren{y-T_i(x)}}_\pbf^2}\nonumber\\
    &=& \sum_{j=1}^m p_j\tfrac{1}{p_j}\norm{\paren{x_j-T'_{j}(x)} - \paren{y_j-T'_{j}(y)}}_{\Ecal_j}^2\nonumber\\
    &=& \norm{\paren{x-T_1(x)}-\paren{y-T_1(y)}}^2.\label{eq:time}
\end{eqnarray}
Combining \eqref{eq:last} with $\tfrac{1-\alphabar}{\alphabar}$ times 
\eqref{eq:time} yields
\begin{eqnarray}
&& \Ebb\ecklam{\norm{\Phi(x, \xi)- \Phi(y,\xi)}_\pbf^2} + 
 \tfrac{1-\alphabar}{\alphabar} \Ebb\ecklam{\psi_\pbf(x,y,\Phi(x,\xi),\Phi(y,\xi))}\nonumber\\
&&\quad =  \norm{T_1(x)-T_1(y)}^2 - \norm{x-y}^2 + \norm{x-y}_\pbf^2 + 
 \tfrac{1-\alphabar}{\alphabar} \norm{\paren{x-T_1(x)}-\paren{y-T_1(y)}}^2.
 \label{eq:Bitter}
\end{eqnarray}
Now by assumption \eqref{t:afneie b}, $T_1$ is a$\alpha$-fne with constant 
$\alphabar$ and violation no greater than $\epsilonbar$ on $ G$.  
Therefore \eqref{eq:Bitter} is bounded by 
\begin{eqnarray}
\Ebb\ecklam{\norm{\Phi(x, \xi)- \Phi(y,\xi)}_\pbf^2} + 
 \tfrac{1-\alphabar}{\alphabar} \Ebb\ecklam{\psi_\pbf(x,y,\Phi(x,\xi),\Phi(y,\xi))}
 &\leq& \epsilonbar\norm{x-y}^2 + \norm{x-y}_\pbf^2\nonumber\\
&\leq& (1+\pbar\epsilonbar)\norm{x-y}_\pbf^2
 \label{eq:Sweet}
\end{eqnarray}
for all $x,y\in G$ as claimed.  
\end{proof}

Following \cite{HerLukStu22b}, we lift these notions to the analogous regularity of Markov operators on 
the space of probability measures. 
Let  $\mathcal{P}$ be the Markov operator with transition kernel
\[
  (x\in G\subset\Ecal) (A\in
  \mathcal{B}(G)) \qquad p(x,A) \equiv \mathbb{P}(\Phi(x,\xi) \in A)
\]
where $\xi$ is an $\Ibb$-valued random variable and 
$\map{\Phi}{G\times  \Ibb}{G}$ is a measurable update function. 
The Markov operator is said to be \emph{pointwise
	  almost nonexpansive in measure at $\mu_0\in \mathscr{P}(G)$} on $\mathscr{P}(G)$, 
	  abbreviated {\em pointwise almost nonexpansive in measure}, whenever 
    \begin{align}\label{eq:paneim}
		\exists \epsilon\in [0,1):\quad d_{W_{2,\pbf}}(\mu\Pcal, \mu_0\Pcal) \le \sqrt{1+\epsilon}\,
		d_{W_{2,\pbf}}(\mu, \mu_0), \qquad \forall \mu\in \mathscr{P}(G).
    \end{align}
    When the above inequality holds for all $\mu_0\in \mathscr{P}(G)$ then
    $\Pcal$ is said to be {\em almost nonexpansive  in measure on
      $\mathscr{P}(G)$}.  As before, the violation is a value of  $\epsilon$ for which 
  \eqref{eq:paneim} holds. When the violation is $0$, the qualifier ``almost'' is dropped. 
  Let $\Ccal_*(\mu_1,\mu_2)$ denote the set of couplings where the distance $d_{W_{2,\pbf}}(\mu_1,\mu_2)$
 is attained (i.e. the optimal couplings between $\mu_1$ and $\mu_2$)
The Markov operator $\Pcal$ is said to be {\em pointwise almost $\alpha$-firmly 
nonexpansive  in measure at $\mu_0\in \mathscr{P}(G)$} on $\mathscr{P}(G)$, 
abbreviated {\em pointwise a$\alpha$-fne in measure}, 
  whenever
  \begin{eqnarray}
&&\exists \epsilon\in [0,1), \alpha\in (0,1): 
\qquad \forall\mu\in \mathscr{P}(G),\forall \gamma\in C_*(\mu, \mu_0)
\nonumber\\
&& d_{W_{2,\pbf}}(\mu\Pcal, \mu_0\Pcal)^2\leq 
(1+\epsilon)d_{W_{2,\pbf}}(\mu, \mu_0)^2 - \nonumber\\
&&\qquad \qquad\qquad \qquad
\tfrac{1-\alpha}{\alpha}\int_{G\times G}\mathbb{E}\left[\psi_\pbf(x,y, \Phi(x,\xi), \Phi(y,\xi))\right] \gamma(dx, dy).
	  \label{eq:paafne i.m.}
\end{eqnarray}
When the above inequality holds for all $\mu_0\in \mathscr{P}(G)$ then $\Pcal$ is said to be 
{\em a$\alpha$-fne in measure on $\mathscr{P}(G)$}.  The 
violation is a value of $\epsilon$ for which \eqref{eq:paafne i.m.} holds. 
When the violation is $0$, the qualifier ``almost'' is dropped and the abbreviation 
{\em $\alpha$-fne in measure} is employed.  The notions above were defined in \cite[Definition 2.8]{HerLukStu22b}
on more general metric spaces.  

\begin{prop}[Proposition 2.10, \cite{HerLukStu22b}]\label{thm:Tafne in exp 2 pafne of P}
  Let $G\subset\Ecal$, let 
   $\map{\Phi}{G\times \Ibb}{G}$ be given by $\Phi(x,i) = T_{i}x$
   and let $\psi_\pbf$ be defined by \eqref{eq:psi_p}. 
   Denote  by $\mathcal{P}$  the Markov operator with update function $\Phi$ and 
   transition kernel $p$ defined by
  \eqref{eq:trans kernel}.
  If $~\Phi$  is a$\alpha$-fne in expectation on $G$ with
  constant $\alpha\in (0,1)$ and violation $\epsilon\in [0,1)$, then the Markov operator 
  $\mathcal{P}$ is a$\alpha$-fne 
  in measure on $\mathscr{P}_2(G)$ 
  with constant $\alpha$ and violation at most $\epsilon$, that is, $\Pcal$ 
  satisfies 
  \begin{eqnarray}
d_{W_{2,\pbf}}^{2}(\mu_1\mathcal{P}, \mu_2\mathcal{P})&\le&    
   (1+\epsilon)d_{W_{2,\pbf}}^{2}(\mu_1, \mu_2) -  
   \tfrac{1-\alpha}{\alpha}  \int_{G\times G}
   \mathbb{E}\ecklam{\psi_\pbf(x,y, \Phi(x,\xi), \Phi(y,\xi))}\ \gamma(dx, dy)
   \nonumber\\
\label{eq:alphfne meas}&&   
\qquad\qquad\qquad\qquad   
 \forall \mu_2, \mu_1\in \mathscr{P}_2(G), \ \forall \gamma\in C_*(\mu_1,\mu_2). 
  \end{eqnarray}
\end{prop}

\begin{thm}[stochastic block iterations]\label{t:harvest1}
 Let the single-valued self-mappings $\{T_i\}_{i\in \Ibb}$ on the convex 
 subset $G\subset \Ecal$ 
satisfy
\begin{enumerate}[(a)]
 \item\label{t:fneie a} $T_i$ is the identity mapping on $\Ecal_{M_i^{\circ}}$;
 \item\label{t:fneie b} $T_1$ is a$\alpha$-fne on $ G$ with constant $\alphabar$ and 
 violation no greater than $\epsilonbar$.
\end{enumerate}
Then the Markov operator $\Pcal$ with update function 
$\Phi$ is a$\alpha$-fne in measure 
with constant $\alphabar$ and violation at most $\pbar\epsilonbar$.
\end{thm}
\begin{proof}
This is an immediate consequence of Theorem \ref{t:afneie}
 and Proposition \ref{thm:Tafne in exp 2 pafne of P}.
\end{proof}

Note also that, since $\psi_\pbf$ is nonnegative, $T_i$ is also 
almost nonexpansive in expectation whenever $T_1$ is a$\alpha$-fne;  the corresponding 
Markov operator is almost nonexpansive in 
measure with the corresponding violation whenever conditions \eqref{t:fneie a}-\eqref{t:fneie b} 
of Theorem \ref{t:harvest1}
are satisfied.  

In preparation for the next refinements, following \cite{HerLukStu22b} 
we lift the weighted transport discrepancy 
$\psi_\pbf$ to the corresponding {\em invariant Markov transport discrepancy} 
$\map{\Psi}{\mathscr{P}(G)}{\mathbb{R}_+}\cup\{+\infty\}$ on the subset $G\subset\Ecal$ 
 defined by 
 \begin{equation}\label{eq:Psi}
\Psi(\mu)\equiv \inf_{\pi\in\inv\mathcal{P}}\inf_{\gamma\in \Ccal_*(\mu,\pi)}
\left(\int_{G\times G}
\mathbb{E}\left[\psi_\pbf(x,y, T_\xi x, T_\xi y)\right]\ \gamma(dx, dy)\right)^{1/2}.
 \end{equation}
 It is not guaranteed  that both $\inv\mathcal{P}$ and $\Ccal_*(\mu,\pi)$ are nonempty;  
 when at least one of these is empty $\Psi(\mu)\equiv +\infty$. 
It is clear that $\Psi(\pi)=0$ for any 
$\pi \in \inv\Pcal$.

\subsection{Special Case: consistent stochastic feasibility}\label{s:consistent}
The stochastic fixed point problem \eqref{eq:stoch_fix_probl}  is called {\em consistent} in 
\cite{HerLukStu19a, HerLukStu22a, HerLukStu22b} when, for some closed subset $G\subset\Ecal$,  
\begin{align}
  \label{eq:stoch_feas_probl}
C := \set{x \in G}{\mathbb{P}(x = T_{\xi}x) = 1}\neq\emptyset.
\end{align}
In this case, the notions developed above can be sharpened.

Recall that a {\em paracontraction} is a continuous mapping $\map{T}{ G}{ G}$
possessing fixed points 
that satisfies
\[
 \|T(x)-y\|< \|x-y\|\quad\forall y\in\Fix T, \forall x\in G\setminus\Fix T.
\]
Any $\alpha$-fne mapping on a Euclidean space, for example, is a paracontraction.  

The notion of paracontractions extends to random function iterations for consistent stochastic
feasibility.  Continuous self-mappings $\map{T_i}{G}{G}$ ($i\in\Ibb$) are 
{\em paracontractions in expectation} with respect to the 
weighted norm $\|z\|_\pbf$ 
whenever  
\begin{equation}\label{e:pie}
C\neq\emptyset\und  \Ebb\ecklam{\|T_\xi x-y\|_\pbf}< \|x-y\|_\pbf\quad\forall y\in C, \forall x\in G\setminus\Fix T.
\end{equation}
The next result shows that, for consistent stochastic feasibility,  
collections of mappings $T_i$ defined in Theorem \ref{t:afneie}
with $\epsilonbar=0$ are paracontractions in expectation.  

\begin{cor}[paracontractions in expectation]\label{t:para}
 Let the single-valued self-mappings $\{T_i\}_{i\in \Ibb}$ on $ G$ 
satisfy
\begin{enumerate}[(a)]
 \item\label{t:para a} $T_i$ is the identity mapping on $\Ecal_{M_i^{\circ}}$;
 \item\label{t:para b} for every $z\in M_i^\circ$, $T_i$ is $\alpha$-fne on 
 $ G_{M_i}\bigoplus\{z\}_{M_i^\circ}$ with constant $\alphabar$ 
 for all $i$;
 \item $C := \set{x \in  G}{\mathbb{P}(x = T_{\xi}x) = 1}\neq\emptyset$.
\end{enumerate}
Then 
the mapping $\map{\Phi}{ G\times \Ibb}{ G}$ given by 
$\Phi(x,i) = T_{i}x$ is a paracontraction in expectation:
\begin{equation}\label{eq:Phi-Pcal para}
 \Ebb\ecklam{\norm{\Phi(x, \xi)- \Phi(y,\xi)}_\pbf^2}< 
 \|x-y\|_\pbf^2 \quad\forall x\in G\setminus C, \forall y\in C.
\end{equation}
\end{cor}
\begin{proof}
Note that 
$\psi_\pbf$ takes the value $0$
only when $x$ and $y$ are both in $\Fix T_i$; hence, for all $y\in C$ 
\begin{equation}\label{eq:Phi afneie}
 \Ebb\ecklam{\norm{T_{\xi}(x)- T_{\xi}(y)}_\pbf^2}<
 \|x-y\|_\pbf^2\quad\forall x\in G\setminus C.
\end{equation}
\end{proof}

To show the analogous result for the Markov operator $\Pcal$ requires more work. 
A Markov operator is a paracontraction with respect to the weighted Wasserstein
metric $d_{W_{2,M}}$ whenever   
\begin{equation}\label{e:Mpie}
\inv\Pcal\neq\emptyset\und  
d_{W_{2,M}}(\mu\Pcal, \pi)< d_{W_{2,M}}(\mu, \pi)\quad\forall \pi\in \inv\Pcal, 
 \forall \mu\in \mathscr{P_2}(G)\setminus\inv\Pcal.
\end{equation}
In the case of consistent stochastic feasibility, the invariant Markov transport discrepancy 
reduces to a very simple form. 
Indeed, note first of all that a $\delta$-distribution centered on any point $x\in C$ 
is invariant with respect to $\Pcal$ so the set of invariant measures supported on $C$,  
\begin{align}
  \label{eq:stoch_feas_probl2}
\mathscr{C} := \set{\mu \in \inv\Pcal}{\Supp\mu\subset C},
\end{align}
is nonempty whenever $C$ is.  
Now suppose $\pi\in\mathscr{C}$.  Then 
$y= T_\xi y$ almost surely whenever $y\in\Supp\pi$ and \eqref{eq:time} yields
 \begin{eqnarray}
\inf_{\gamma\in \Ccal_*(\mu,\pi)}
\left(\int_{G\times G}
\mathbb{E}\left[\psi_\pbf(x,y, T_\xi x, T_\xi y)\right]\ \gamma(dx, dy)\right)^{1/2}
&=& 
\inf_{\gamma\in \Ccal_*(\mu,\pi)}
\left(\int_{G\times G}
\mathbb{E}\left[\|x - T_\xi x\|_\pbf^2\right]\ \gamma(dx, dy)\right)^{1/2}\nonumber\\
&=& 
\left(\int_{G}
\mathbb{E}\left[\|x - T_\xi x\|_\pbf^2\right]\ \mu(dx)\right)^{1/2}\nonumber\\
&=& \left(\int_{G}
\|x - T_1 x\|^2\, \mu(dx)\right)^{1/2}\quad\forall\pi\in\mathscr{C}.
\label{eq:StoPBForBS}
 \end{eqnarray}
 Thus the invariant Markov transport discrepancy defined in \eqref{eq:Psi} has 
 the following simple upper bound:
  \begin{eqnarray}
\Psi(\mu)&\equiv& \inf_{\pi\in\inv\Pcal}\inf_{\gamma\in \Ccal_*(\mu,\pi)}
\left(\int_{G\times G}
\mathbb{E}\left[\psi_\pbf(x,y, T_\xi x, T_\xi y)\right]\ \gamma(dx, dy)\right)^{1/2}
\nonumber\\
&\leq&
\inf_{\pi\in\mathscr{C}}\inf_{\gamma\in \Ccal_*(\mu,\pi)}
\left(\int_{G\times G}
\mathbb{E}\left[\psi_\pbf(x,y, T_\xi x, T_\xi y)\right]\ \gamma(dx, dy)\right)^{1/2}
\nonumber\\
&=&
\left(\int_{G}
\|x - T_1 x\|^2\, \mu(dx)\right)^{1/2},
\label{eq:Psi_StoPBForBS}
 \end{eqnarray}
 where the last equality follows from \eqref{eq:StoPBForBS}. Inequality 
 \eqref{eq:Psi_StoPBForBS} is tight for all $\mu$ supported on $C$, so clearly 
 $\mu\in\mathscr{C}$ implies that $\Psi(\mu)=0$.  
 On the other hand, if $\Psi(\mu)=0$ implies that 
$\Supp\mu\subset C$, then $\mathscr{C}=\inv\Pcal$ and 
\eqref{eq:Psi_StoPBForBS} holds with equality for all $\mu$.   
This holds, in particular, 
when $T_i$ is a paracontraction in expectation (see \cite[Lemma 3.3]{HerLukStu19a}
and Theorem \ref{t:invMeasforParacontra} below). 

Let's assume, then, that  $\Psi(\mu)=0$ if and only if  $\Supp\mu\subset C$.  
Then  
\[
d_{W_{2,\pbf}}(\mu,\inv\Pcal) = \paren{\int_G \inf_{z\in C}\|x-z\|_\pbf^2 \mu(d x)}^{1/2}, 
\]
and \eqref{eq:Psi_StoPBForBS} holds with equality, so 
\begin{equation}\label{e:Tricky}
d_{W_{2,\pbf}}(\mu,\inv\Pcal) = d_{W_{2,\pbf}}(\mu,\Psi^{-1}(0)) = 
\paren{\int_G \inf_{z\in C}\|x-z\|_\pbf^2 \mu(d x)}^{1/2}. 
\end{equation}

\begin{thm}[Markov operators of paracontractions in expectation] \label{t:invMeasforParacontra} 
Let $G\subset\Ecal$ be closed.  If the continuous self-mappings $\map{T_i}{G}{G}$ ($i\in\Ibb$) defined by 
  \eqref{e:T blockwise} are 
  paracontractions in expectation on $G$ with respect to the weighted norm $\|\cdot\|_\pbf$ 
  defined by \eqref{e:norm pbf}, then
  \begin{enumerate}[(i)]
   \item\label{t:invMeasforParacontra i} the associated Markov operator 
   $\Pcal$ is a paracontraction with respect to $d_{W_{2,\pbf}}$;
   \item\label{t:invMeasforParacontra ii} if $G$ is bounded, the set of invariant measures for $\mathcal{P}$ is 
   $\set{\pi \in \mathscr{P}(G)}{\Supp \pi \subset C}$;
   \item\label{t:invMeasforParacontra iii} if $G$ is bounded,
   \begin{subequations}\label{e:pointwise error bound}
    \begin{eqnarray}\label{e:simple psi}
      (\forall x\in G)\quad \Psi(\delta_x) &=& \norm{x-T_1(x)}\quad \und \\
     \frac{1}{\sqrt{\pbar}}\inf_{z\in C}\|x-z\|&\leq& \inf_{z\in C}\|x-z\|_\pbf 
      =  d_{W_{2,\pbf}}(\delta_x,\inv\Pcal) = d_{W_{2,\pbf}}\paren{\delta_x,\Psi^{-1}(0)}.
    \end{eqnarray}
    \end{subequations}
  \end{enumerate}
\end{thm}
\begin{proof} 
\eqref{t:invMeasforParacontra i}.
For a random variable $X\sim \mu$, we have $T_\xi X = \Phi(X, \xi)\sim \mu\mathcal{P}$,
and for a random variable $Y\sim \pi\in \inv \mathcal{P}$
we have $T_\xi Y =\Phi(Y, \xi)\sim \pi\mathcal{P}= \pi$, so
\begin{eqnarray}
 d_{W_{2,\pbf}}(\mu\Pcal, \pi) &=& 
 \paren{\inf_{\gamma\in \Ccal(\mu\Pcal,\pi)}\int_{G\times G} \|x^+-y\|_\pbf^2\gamma(dx^+, dy)}^{1/2}
 \nonumber\\
 &\leq& \paren{\inf_{\gamma\in \Ccal(\mu,\pi)}\int_{G\times G} \Ebb\ecklam{\|T_\xi x-y\|_\pbf^2}\gamma(dx, dy)}^{1/2}
 \nonumber\\
 &<& \paren{\inf_{\gamma\in \Ccal(\mu,\pi)}\int_{G\times G} \|x-y\|_\pbf^2\gamma(dx, dy)}^{1/2}
 \nonumber\\
 &=& d_{W_{2,\pbf}}(\mu, \pi) \qquad\forall \pi\in\inv\Pcal, \forall\mu\in\mathscr{P}_2(G)\setminus\inv\Pcal,
\end{eqnarray}
where the last inequality follows from the assumption that 
$T_i$ defined by \eqref{e:T blockwise} are 
  a paracontractions in expectation with respect to the weighted norm $\|\cdot\|_\pbf$.
  This establishes that $\Pcal$ is a paracontraction in the $d_{W_{2,\pbf}}$ metric as claimed.

\eqref{t:invMeasforParacontra ii}. Our proof  
follows the proof of \cite[Lemma 3.3]{HerLukStu19a}.
  It is clear that $\pi \in \mathscr{P}(G)$ with $\Supp \pi \subset C\subset G$
  is invariant, since $p(x,\{x\}) = \mathbb{P}(T_{\xi}x \in \{x\}) =
  \mathbb{P}(x \in \Fix T_{\xi}) = 1$ for all $x \in C$ and hence
  $\pi\mathcal{P}(A) = \int_{C}p(x,A) \pi(\dd{x}) = \pi(A)$ for all $A
  \in \mathcal{B}(G)$.

  Suppose, on the other hand, that $ \Supp \pi
  \setminus C \neq \emptyset$ for some $\pi\in\inv\Pcal$ with $\Supp\pi\subset G$.  
  Then due to compactness of $\Supp \pi$ (it is
  closed in the compact set $G$) we can find $s \in \Supp \pi$ maximizing the
  continuous function $d(\cdot,C)\equiv \inf_{z\in C}\|\cdot - z\|$ on $G$. So $d_{\text{max}} =
  d(s,C) > 0$. We show that this leads only to contradictions, so 
  the assumption of the existence of such a $\pi$ must be false.  

  Define the set of points being more than $d_{\text{max}}-\epsilon$
  away from $C$:
  \begin{align*}
    K(\epsilon) := \set{x \in G}{ d(x,C) > d_{\text{max}} -
      \epsilon }, \qquad \epsilon \in (0,d_{\text{max}}).
  \end{align*}
  This set is measurable, i.e.\ $K(\epsilon) \in \mathcal{B}(G)$,
  because it is open. Let $M(\epsilon)$ be the event in the sigma algebra
  $\mathcal{F}$, that $T_{\xi}s$ is at least $\epsilon$ closer to $C$ than $s$,
  i.e.\
  \begin{align*}
    M(\epsilon) := \set{ \omega \in \Omega}{ d(T_{\xi(\omega)}s,C) \le d_{\text{max}} - \epsilon}.
  \end{align*}
  There are two possibilities, either there is an $\epsilon \in
  (0,d_{\text{max}})$ with $\mathbb{P}(M(\epsilon)) >0$ or no such
  $\epsilon$ exists.  In the latter case we have $\Ebb\ecklam{d(T_{\xi}s,C)} =
  d_{\text{max}} = d(s,C)$ since 
  $T_{i}$ is a paracontraction in expectation. By compactness of $C$ there exists $c \in C$ such that
  $0<d_{\text{max}} = \|s-c\|$.  Hence the probability of the set of
  $\omega\in\Omega$ such that $s \not \in \Fix T_{\xi(\omega)}$ is
  positive and so $\Ebb\ecklam{d(T_{\xi(\omega)}s, C)}
  \le \Ebb\ecklam{\|T_{\xi(\omega)}s-c\|} < \|s - c\|$ - a contradiction.

  Suppose next that there is an $\epsilon \in (0,d_{\text{max}})$
  with $\mathbb{P}(M(\epsilon)) >0$.  In view of continuity of the
  mappings $T_{i}$ around $s$, $i \in \Ibb$, define
  \begin{align*}
    A_{n} : = \set{\omega \in M(\epsilon)}%
    { \|(T_{\xi(\omega)}x - T_{\xi(\omega)} s\|
    \le \tfrac{\epsilon}{2} \quad \forall
      x\in\mathbb{B}(s,\tfrac{1}{n})}\quad (n \in \mathbb{N}).
  \end{align*}
  It holds that $A_{n} \subset A_{n+1}$ and $\mathbb{P}(\bigcup_{n}
  A_{n}) = \mathbb{P}(M(\epsilon))$. So in particular there is an $m
  \in \mathbb{N}$, $m \ge 2/\epsilon$ with $\mathbb{P}(A_{m}) > 0$.
  For all $x \in \mathbb{B}(s,\tfrac{1}{m})$ and all $\omega\in A_{m}$
  we have
  \begin{align*}
    d(T_{\xi(\omega)} x, C) \le 
    \|T_{\xi(\omega)}x - T_{\xi(\omega)} s\| + d(T_{\xi(\omega)} s,C) 
    \le d_{\text{max}}
    - \frac{\epsilon}{2},
  \end{align*}
  which means $T_{\xi(\omega)}x \in G \setminus
  K(\tfrac{\epsilon}{2})$.  Hence, in particular we conclude that
  \[
  p(x,K(\tfrac{\epsilon}{2})) < 1 \quad\forall x \in
  \mathbb{B}(s,\tfrac{1}{m}).
  \]
  Since $p(x,K(\epsilon)) = 0$ for $x \in G$ with $d(x,C) \le
  d_{\text{max}} - \epsilon$ by the assumption that $T_i$ is 
  a paracontraction in expectation, it holds by
  invariance of $\pi$ that
  \begin{align*}
    \pi(K(\epsilon)) = \int_{G} p(x,K(\epsilon)) \pi(\dd{x}) =
    \int_{K(\epsilon)} p(x,K(\epsilon)) \pi(\dd{x}).
  \end{align*}
  It follows, then, that
  \begin{align*}
    \pi(K(\tfrac{\epsilon}{2})) &= \int_{K(\tfrac{\epsilon}{2})}
    p(x,K(\tfrac{\epsilon}{2})) \pi(\dd{x}) \\ &=
    \int_{\mathbb{B}(s,\tfrac{1}{ m})}p(x, K(\tfrac{ \epsilon}{2}) )
    \pi( \dd{x}) + \int_{K(\tfrac{\epsilon}{2}) \setminus
      \mathbb{B}(s,\tfrac{1}{ m})} p(x,K(\tfrac{\epsilon}{2}))
    \pi(\dd{x}) \\ &< \pi( \mathbb{B}(s ,\tfrac{1}{m})) +
    \pi(K(\tfrac{\epsilon}{2}) \setminus \mathbb{B}(s, \tfrac{1}{
      m}))= \pi(K( \tfrac{ \epsilon}{2}))
  \end{align*}
  which leads again to a contradiction.  So the assumption that $\Supp
  \pi\setminus C \neq \emptyset$ is false, i.e.\ $\Supp \pi \subset C$
  as claimed.

\eqref{t:invMeasforParacontra iii}. By part \eqref{t:invMeasforParacontra ii}, 
$\inv\Pcal = \set{\pi\in\mathscr{P}(G)}{\Supp\pi\subset C}$, 
so \eqref{eq:Psi_StoPBForBS}
holds with equality, and $\Psi(\mu)=0$
if and only if $\Supp \mu\subset C$, hence writing \eqref{e:Tricky} pointwise 
(i.e., for $\mu=\delta_x$) reduces the expression to
\[
\inf_{z\in C}\|x-z\|_\pbf  = d_{W_{2,\pbf}}(\delta_x,\inv\Pcal) = d_{W_{2,\pbf}}(\delta_x,\Psi^{-1}(0)).
\]
The representation \eqref{e:pointwise error bound} then follows from $\pbar\equiv\max_j\{p_j\}$.
  \end{proof}

\section{Convergence}\label{s:convergence}
Contractive Markov operators have been extensively, almost exclusively, studied.   
When the update function $\Phi$ is a 
contraction in expectation, then \cite[Theorem 2.12]{HerLukStu22b} shows that the 
corresponding Markov operator $\Pcal$ 
is $\alpha$-fne, and the sequence of measures $(\mu_k)$ converges Q-linearly (geometrically) 
to an invariant measure from any starting measure $\mu_0\in\mathscr{P}(\Ecal)$.
When the mappings $T_i$ are only $\alpha$-firmly nonexpansive on $\Ecal$, then $\mu_k$
converges in the Prokhorov-Levi metric to an invariant measure from any initial measure 
\cite[Theorem 2.9]{HerLukStu22a}.   To obtain generic (weak) convergence of the iterates 
$\mu_k$ one must show that the 
sequence is {\em tight}.  This has been established for Markov operators with 
nonexpansive update functions \cite[Lemma 3.19]{HerLukStu22a}.  We skirt a study of whether 
tightness can be established under the assumption that the update functions 
$\Phi(x,i)$ are only nonexpansive in 
expectation;  we suspect, however, that this is not the case.  

\subsection{Generic proto-convergence}
We establish a few properties that are cornerstones of a generic global convergence 
analysis.  In particular, we show that 
when the Markov operator is $\alpha$-fne (which, as shown above, does not require that 
all the mappings $T_i$ be $\alpha$-fne) this property together with an additional 
assumption about the decay of the invariant Markov transport discrepancy yields 
boundedness and {\em asymptotic regularity} of the sequence of measures.  
\begin{propn}[asymptotic regularity]\label{t:asymp reg}
   Let the Markov operator $\map{\Pcal}{\mathscr{P}_2(\Ecal)}{\mathscr{P}_2(\Ecal)}$
   with update functions $\Phi(x, i)$ possess at least one
   invariant measure and be pointwise $\alpha$-fne in measure at all $\pi\in\inv\Pcal$.  If the 
   invariant Markov transport discrepancy satisfies 
   \begin{equation}\label{e:Lyapunov}
    \exists c>0: \quad \Psi(\mu)\geq c d_{W_{2,\pbf}}(\mu,\mu\Pcal)\quad\forall\mu\in\mathscr{P}_2(\Ecal),
   \end{equation}
   then the sequence $(\mu_k)_{k\in\Nbb}$ defined by $\mu_{k+1}=\mu_k\Pcal$ for any $\mu_0\in\mathscr{P}_2(\Ecal)$ 
   is bounded and asymptotically 
   regular, i.e. satisfies $d_{W_{2,\pbf}}(\mu_k,\mu_{k+1}) \to  0$.
\end{propn}
\begin{proof}
Note that \eqref{e:Lyapunov} implies that there is a $c>0$ such that 
  \begin{eqnarray*}
    c^2 d_{W_{2,\pbf}}(\mu,\mu\Pcal)^2\leq \int_{\Ecal\times\Ecal}\Ebb\ecklam{\psi_\pbf(x,y, \Phi(x,\xi), \Phi(y,\xi))}
    \gamma(dx, dy)\quad\forall \pi\in\inv\Pcal,~\forall\gamma\in C_*(\pi,\mu).
   \end{eqnarray*}
 This together with the assumption that $\Pcal$ is $\alpha$-fne yields
  \begin{eqnarray}
    0\leq d_{W_{2,\pbf}}(\mu\Pcal, \pi)^2&\leq& d_{W_{2,\pbf}}(\mu, \pi)^2 - 
    \frac{1-\alpha}{\alpha}\int_{\Ecal\times\Ecal}\Ebb\ecklam{\psi_\pbf(x,y, \Phi(x,\xi), \Phi(x,\xi))}
    \gamma(dx, dy)\nonumber\\
    &\leq& d_{W_{2,\pbf}}(\mu, \pi)^2 - 
    \frac{1-\alpha}{\alpha}c^2 d_{W_{2,\pbf}}(\mu,\mu\Pcal)^2\quad\forall \pi\in\inv\Pcal,~\forall\gamma\in C_*(\pi,\mu).
    \label{e:stay high}
   \end{eqnarray}
Applying \eqref{e:stay high} to 
the sequence of measures generated by $\mu_{k+1}=\mu_k\Pcal$ with 
$\mu_0\in\mathscr{P}_2(\Ecal)$
yields
\[
 \frac{1-\alpha}{\alpha}c^2 \sum_{k=1}^N d_{W_{2,\pbf}}(\mu_k,\mu_{k+1})^2\leq d_{W_{2,\pbf}}(\mu_0,\pi)^2
\quad \forall \pi\in\inv\Pcal, \forall N\in\Nbb.
\]
Letting $N\to\infty$ establishes that the left hand side is summable, hence 
$\liminf d_{W_{2,\pbf}}(\mu_k,\mu_{k+1}) = 0$.  But $\Pcal$ is also pointwise nonexpansive 
at all $\pi\in\inv\Pcal$ since it is pointwise $\alpha$-fne there, 
so $d_{W_{2,\pbf}}(\mu_k,\pi)\leq d_{W_{2,\pbf}}(\mu_0,\pi)$
for all $k$ and $d_{W_{2,\pbf}}(\mu_k,\mu_{k+1}) \to  0$; i.e. the sequence is bounded and 
asymptotically regular as claimed. 
\end{proof}

In the next section we 
pursue a quantitative local convergence analysis under the assumption of {\em metric subregularity}
of the invariant Markov transport discrepancy. 

\subsection{Metric subregularity of the invariant Markov transport discrepancy, convergence and rates}
\label{s:msr mtd}
Recall the inverse mapping 
$\Psi^{-1}(y)\equiv \set{\mu}{\Psi(\mu)=y}$, which 
clearly can be set-valued.  It is important to keep in mind 
that an invariant measure need not correspond to a fixed point of any 
individual mapping $T_i$, unless these have common fixed points.  See 
\cite{HerLukStu22a, HerLukStu22b} instances of this.  
We require that the invariant Markov transport discrepancy $\Psi$ takes the value $0$ at 
$\mu$ if and only if $\mu\in\inv\Pcal$, and is 
{\em gauge metrically subregular for $0$ relative 
to $\mathscr{P}_2(G)$ on $\mathscr{P}_2(G)$}:
\begin{equation}\label{e:metricregularity}
d_{W_{2,\pbf}}(\mu,\inv\Pcal) = d_{W_{2,\pbf}}(\mu,\Psi^{-1}(0))\leq \rho(\Psi(\mu))
\quad \forall \mu\in \mathscr{P}_2(G).
\end{equation}
Here  $d_{W_{2,\pbf}}(\mu,\inv\Pcal) = \inf_{\pi\in\inv\Pcal}d_{W_{2,\pbf}}(\mu,\pi)$,
and  $\rho:[0,\infty) \to [0,\infty)$ is a \textit{gauge function}: it is 
continuous, strictly increasing 
with $\rho(0)=0$, and $\lim_{t\to \infty}\rho(t)=\infty$. 
The gauge of metric subregularity $\rho$  is constructed 
implicitly from another 
nonnegative function $\map{\theta_{\tau,\epsilon}}{[0,\infty)}{[0,\infty)}$
with parameters $\tau>0$ and $\epsilon\geq 0$ satisfying  
\begin{eqnarray}\label{eq:theta_tau_eps}
(i)~ \theta_{\tau,\epsilon}(0)=0; \quad (ii)~ 0<\theta_{\tau,\epsilon}(t)<t ~\forall t\in(0,\tbar]
\mbox{ for some }\tbar>0\end{eqnarray}
and 
\begin{equation}\label{eq:gauge}
 \rho\paren{\paren{\frac{(1+\epsilon)t^2-\paren{\theta_{\tau,\epsilon}(t)}^2}{\tau}}^{1/2}}=
 t\quad\iff\quad
 \theta_{\tau,\epsilon}(t) = \paren{(1+\epsilon)t^2 - \tau\paren{\rho^{-1}(t)}^2}^{1/2}
\end{equation}
for $\tau>0$ fixed.  In the next theorem the parameter $\epsilon$ is 
exactly the violation in a$\alpha$-fne mappings; the parameter $\tau$ is directly computed from the 
constant $\alpha$.  

In preparation for the results that follow, we will require at least one of the 
additional assumptions on $\theta$.
\begin{assumption}\label{ass:msr convergence}
The gauge $\theta_{\tau,\epsilon}$ satisfies \eqref{eq:theta_tau_eps} and 
at least one of the following holds.  
 \begin{enumerate}[(a)]
\item\label{t:msr convergence, necessary sublin} $\theta_{\tau,\epsilon}$ satisfies 
\begin{equation}\label{eq:theta to zero}
    \theta_{\tau,\epsilon}^{(k)}(t)\to 0\mbox{ as }k\to\infty~\forall t\in(0,\tbar),
\end{equation}
 and the sequence $(\mu_k)$ is Fej\'er monotone with respect to $\inv\mathcal{P}\cap \mathscr{P}_2(G)$, i.e.
\begin{equation}\label{eq:Fejer}
d_{W_{2,\pbf}}\paren{\mu_{k+1},\, \pi}\leq d_{W_{2,\pbf}}(\mu_k, \pi) \quad \forall k\in\Nbb, 
\forall \pi\in \inv\mathcal{P}\cap\mathscr{P}_2(G);
\end{equation}
\item\label{t:msr convergence, necessary lin+} 
$\theta_{\tau,\epsilon}$ satisfies 
\begin{equation}\label{eq:theta summable}
    \sum_{j=1}^\infty\theta_{\tau,\epsilon}^{(j)}(t)<\infty~\forall t\in(0,\tbar)
\end{equation}
where $\theta_{\tau,\epsilon}^{(j)}$ denotes the $j$-times composition of $\theta_{\tau,\epsilon}$. 
\end{enumerate}
\end{assumption}

In the case of linear metric subregularity 
this becomes 
\[
\rho(t)=\kappa t\quad\iff\quad  
\theta_{\tau, \epsilon}(t)=\paren{(1+\epsilon)-\frac{\tau}{\kappa^2}}^{1/2}t\quad (\kappa\geq \sqrt{\tfrac{\tau}{(1+\epsilon)}}).
\]  
The condition $\kappa\geq \sqrt{\tfrac{\tau}{(1+\epsilon)}}$ is not a real restriction since, if 
\eqref{e:metricregularity} is satisfied for some $\kappa'>0$, then it is satisfied
for all $\kappa\geq \kappa'$. The conditions in \eqref{eq:theta_tau_eps} in this 
case simplify to $\theta_{\tau, \epsilon}(t)=\gamma t$ where 
\begin{equation}\label{eq:theta linear}
 0< \gamma\equiv 1+\epsilon-\frac{\tau}{\kappa^2}<1\quad\iff\quad 
\sqrt{\tfrac{\tau}{(1+\epsilon)}}\leq  \kappa\leq \sqrt{\tfrac{\tau}{\epsilon}}.
\end{equation}
In other words, $\theta_{\tau, \epsilon}(t)$ satisfies 
Assumption \ref{ass:msr convergence}\eqref{t:msr convergence, necessary lin+}. 
The weaker Assumption \ref{ass:msr convergence}\eqref{t:msr convergence, necessary sublin}
is used to characterize sublinear convergence.

\begin{thm}[convergence rates]\label{t:msr convergence} 
  Let $G\subset\Ecal$ be compact.  Let $\map{T_i}{G}{G}$ satisfy the assumptions of  
  Theorem \ref{t:afneie} for all $i\in \Ibb$.
  Assume furthermore that there is at least one 
  $\pi \in\inv\mathcal{P}\cap \mathscr{P}_2(G)$  where $\mathcal{P}$ is the 
  Markov operator associated with $T_i$.  If, in addition,  $\Psi$ satisfies 
  \eqref{e:metricregularity} with gauge $\rho$ given implicitly by \eqref{eq:gauge} 
  in terms of $\theta_{\tau,\epsilon}$ where 
  $\tau=(1-\alphabar)/\alphabar$, $\epsilon=\pbar\epsilonbar$ 
  as in Theorem \ref{t:afneie},  then for any $\mu_0\in \mathscr{P}_2(G)$ 
  the distributions $\mu_k$ of the iterates of Algorithm \ref{algo:sbi} 
  satisfy  
  \begin{equation}\label{eq:gauge convergence}
                d_{W_{2,\pbf}}\paren{\mu_{k+1},\inv\mathcal{P}}
                \leq \theta_{\tau,\epsilon}\paren{d_{W_{2,\pbf}}\paren{\mu_k,\inv\mathcal{P}}} 
                \quad \forall k \in \mathbb{N}.
  \end{equation}%

  In addition, let $\tau$ and $\epsilon$ be such that $\theta_{\tau,\epsilon}$ satisfies 
  \eqref{eq:theta_tau_eps} where 
$t_0\equiv d_{W_{2,\pbf}}\paren{\mu_0,\inv\mathcal{P}}<\tbar$ for all 
$\mu_0\in \mathscr{P}_2(G)$,  and let at least one of the conditions in Assumption \ref{ass:msr convergence} hold. 
Then $\mu_k\to \pi^{\mu_0}\in\inv\mathcal{P}\cap\mathscr{P}_2(G)$
 in the $d_{W_{2,\pbf}}$ metric with rate 
 $O\paren{\theta_{\tau, \epsilon}^{(k)}(t_0)}$ in case 
 Assumption \ref{ass:msr convergence}\eqref{t:msr convergence, necessary sublin}
 and with rate  $O(s_k(t_0))$ for 
$s_k(t)\equiv 
\sum_{j=k}^\infty \theta_{\tau, \epsilon}^{(j)}(t)$
in case Assumption \ref{ass:msr convergence}\eqref{t:msr convergence, necessary lin+}.  
\end{thm}
\begin{proof}
First we note that, since $G$ is assumed to be compact, $\mathcal{P}$ is a 
nonempty self-mapping on $\mathscr{P}_{2}(G)$ and  
$\mathscr{P}_2(G)$ is locally compact (\cite[Remark 7.19]{AmbGigSav2005}).
 By Theorem \ref{t:afneie}, the update function $\Phi$ is a$\alpha$-fne in expectation
 with respect to the weighted norm $\|\cdot\|_\pbf$ with constant $\alphabar$ and 
 violation $\pbar\epsilonbar$.  The statement is an extension of \cite[Theorem 2.6]{HerLukStu22b}, 
 which establishes \eqref{eq:gauge convergence} and convergence under 
 Assumption \ref{ass:msr convergence}\eqref{t:msr convergence, necessary lin+}. 
 
 To establish convergence under Assumption \ref{ass:msr convergence}\eqref{t:msr convergence, necessary sublin}, 
we show first that $d_{W_{2,\pbf}}\paren{\mu_k, S}\to 0$ where, to reduce notational 
clutter we define $S\equiv\inv\mathcal{P}\cap\mathscr{P}_2(G)$.
Indeed, let $\pi\in S$ and define 
$d_k^\pi\equiv d_{W_{2,\pbf}}\paren{\mu_k, \pi}$.  Since 
$d_{k+1}^\pi\leq d_k^\pi$ for all $k$, this establishes that the 
sequence $(d_{k}^\pi)_{k\in\Nbb}$ is bounded and monotone non-increasing, therefore convergent. 
Noting that $d_{W_{2,\pbf}}\paren{\mu_k, S}\leq d_k^\pi$
for all $k$ and any fixed $\pi\in S$, this also shows that 
$d_{W_{2,\pbf}}\paren{\mu_k, S}$ converges.  
The inequality \eqref{eq:gauge convergence} only requires assumption \eqref{eq:theta_tau_eps}, and this
together with assumption \eqref{eq:theta to zero} yields
\begin{eqnarray*}
d_{W_{2,\pbf}}\paren{\mu_k, S}\leq \theta^{(k)}(t_0)\to 0\mbox{ as }k\to\infty.
\end{eqnarray*}
Since $\mathscr{P}_2(G)$ is locally compact and $\Pcal$ is Feller since $T_i$ is continuous
for all $i$, $\inv\Pcal$ is 
closed \cite{Hairer2021}; so for every $k\in\Nbb$  
the infimum in  $d_{W_{2,\pbf}}\paren{\mu_k, S}$
is attained  at some $\pi_k$. 
Now, for such a $\pi_k$ we have, again by Fej\'er monotonicity, that
\begin{eqnarray*}
 d(\mu_l, \mu_k)\leq d(\mu_l, \xbar^k)+d(\mu_k, \xbar^k)\leq d(\mu_{l-1}, \xbar^k)+ d(\mu_k, \xbar^k)\leq \cdots\leq 2d(\mu_k, S).
\end{eqnarray*}
Since the right hand side converges to $0$ as $k\to\infty$ this shows that the sequence
is a Cauchy sequence on  
$(\mathscr{P}_{2}(G), W_2)$ -- a separable complete metric space 
\cite[Theorem 6.9]{Villani2008} -- and therefore convergent to some probability measure 
$\pi^{\mu_0}\in \mathscr{P}_{2}(G)$. The Markov operator $\mathcal{P}$ is Feller 
and when a Feller Markov chain converges
in distribution, it does so to an invariant measure:  $\pi^{\mu_0} \in \inv \mathcal{P}$ 
(see \cite[Theorem 1.10]{Hairer2021}).
\end{proof}

Note that 
\[
\forall \epsilonbar>\epsilon, \forall t\in [0, t_0],\quad \theta_{\tau,\epsilonbar}(t)>\theta_{\tau,\epsilon}(t).
\]
It is common in optimization algorithms to encounter mappings whose violation $\epsilon$ can be controlled
by choosing a step length parameter small enough;  the gradient descent operator is just such a mapping.   
This means that, if condition \eqref{eq:theta_tau_eps} and at least one of 
\eqref{t:msr convergence, necessary sublin}
or \eqref{t:msr convergence, necessary lin+} in Assumption \ref{ass:msr convergence} is satisfied for {\em some} $\epsilon$, 
and the violation of the fixed point mappings $T_i$ can be made arbitrarily small, then 
Theorem \ref{t:msr convergence} guarantees convergence
with rate given by either $O(\theta^{(k)}_{\tau,\epsilon})$ in case 
\eqref{t:msr convergence, necessary sublin} or $O(s_k(t_0))$ in case 
\eqref{t:msr convergence, necessary lin+} for small enough step sizes on small enough 
neighborhoods of a fixed point.  

\subsubsection{Special Case: consistent stochastic feasibility}\label{s:consistent2}
Recall that, when $\inv\Pcal = \mathscr{C}$ defined by \eqref{eq:stoch_feas_probl2} 
(which, by Theorem \ref{t:invMeasforParacontra} holds when $\Pcal$ is a paracontraction in measure)
the 
relation \eqref{eq:Psi_StoPBForBS} holds with equality, so 
condition \eqref{e:metricregularity}  
simplifies to 
\begin{eqnarray}
d_{W_{2,\pbf}}(\mu,\inv\Pcal) = d_{W_{2,\pbf}}(\mu,\Psi^{-1}(0))&\leq& \rho(\Psi(\mu))
\quad \forall \mu\in \mathscr{P}_2(G)\nonumber\\
&\iff&\nonumber\\
\paren{\int_G \inf_{z\in C}\|x-z\|_\pbf^2 \mu(d x)}^{1/2} &\leq& 
\rho\paren{\left(\int_{G}
\|x - T_1 x\|^2\, \mu(dx)\right)^{1/2}}
\qquad \forall \mu\in \mathscr{P}_2(G).
\label{e:pre metricregularity StoPBForBS}
\end{eqnarray}
Writing this pointwise (i.e., for $\mu=\delta_x$) reduces the expression to
\begin{eqnarray}
\inf_{z\in C}\|x-z\|_\pbf &\leq& 
\rho\paren{
\|x - T_1 x\|}\qquad\forall x\in G,
\label{e:almost metricregularity StoPBForBS}
\end{eqnarray}
whereby, recalling that $\pbar = \max_j\{p_j\}$, \eqref{e:metricregularity} yields
\begin{eqnarray}
\tfrac{1}{\sqrt{\pbar}} d(x,C)\leq \inf_{z\in C}\|x-z\|_\pbf &\leq& 
    \rho\paren{\norm{x-T_1(x)}}
    \qquad\forall x\in G.\label{eq:Sport day}
\end{eqnarray}
This is recognizable 
as a slight generalization of the error bound studied by 
Luo and Tseng \cite{LuoTseng93}.

The next result shows that, for paracontractions,  
metric subregularity 
is {\em automatically} satisfied by Markov chains that are {\em gauge monotone}
with respect to $\inv\Pcal$.  Let $(X_k)_{k\in\Nbb}$ be a sequence of random variables
on the closed subset $G\subset\Ecal$ generated by Algorithm \ref{algo:sbi}, and let $(\mu_k)_{k\in\Nbb}$ be the 
corresponding sequence of distributions.  
Let $\inv\Pcal$ be nonempty and let the continuous mapping
$\map{\theta}{\Rbb_+}{\Rbb_+}$ satisfy 
\begin{eqnarray}\label{eq:theta}
(i)~ \theta(0)=0; \quad (ii)~ 0<\theta(t)\leq t ~\forall t\in (0,\tbar)\mbox{ for some } \tbar>0.
\end{eqnarray}
This is obviously the same as \eqref{eq:theta_tau_eps} but without the parameters since in this
case $\epsilon=0$ and $\tau$ is just some scaling.
For $t_0\equiv d_{W_{2,\pbf}}\paren{\mu_0, \inv\Pcal}$, the sequence 
 $(\mu_k)_{k\in \Nbb}$ is said to be
 \emph{gauge monotone relative to $\inv\Pcal$ with rate $\theta$}  whenever
 \begin{equation}\label{e:mu-uniform mon}
 d_{W_{2,\pbf}}(\mu_{k+1}, \inv\Pcal)\leq 
 \theta\paren{d_{W_{2,\pbf}}(\mu_k, \inv\Pcal)}\,\forall k\in \Nbb
 \end{equation}
where $\theta$ satisfies \eqref{eq:theta} with $t_0<\tbar$.
The sequence $(\mu_k)_{k\in \Nbb}$ is said to be
 \emph{linearly monotone relative to $\inv\Pcal$} with rate $c$ if 
 \eqref{e:mu-uniform mon} is
 satisfied for $\theta(t)\leq c\cdot t$ for all $t\in [0,t_0]$ and some constant $c\in [0,1]$.

 A  Markov chain $(X_k)_{k\in \Nbb}$ that converges to some law $\pi^{\mu_0}\in \mathscr{P}_2(G)$
is said to converge {\em gauge monotonically} in distribution 
whenever the corresponding
sequence of distributions $(\mu_k)_{k\in\Nbb}$
is gauge monotone with gauge $\theta$ satisfying \eqref{eq:theta} with 
$d_{W_{2,\pbf}}(\mu_0, \inv\Pcal)\leq \tbar$.

\begin{propn}[gauge monotonic paracontractions in measure converge
to invariant measures]
\label{t:rm and qafne to convergence}
Let $G\subset\Ecal$ be compact.  Let the Markov operator corresponding to 
Algorithm \eqref{algo:sbi}, $\map{\Pcal}{\mathscr{P}_2(G)}{\mathscr{P}_2(G)}$, 
be a paracontraction with respect to the metric $d_{W_{2,\pbf}}$.  
For a fixed $\mu_0\in \mathscr{P}_2(G)$, let the sequence of measures 
$(\mu_k)_{k\in\mathbb{N}}$ corresponding to the iterates of Algorithm \ref{algo:sbi} 
be gauge monotone relative to $\inv\Pcal$ 
with rate $\theta$ satisfying \eqref{eq:theta} 
 where $t_0\equiv d_{W_{2,\pbf}}\paren{\mu_0, \inv\Pcal}<\tbar$. 
Suppose furthermore that at least one of the conditions \eqref{t:msr convergence, necessary sublin}
or \eqref{t:msr convergence, necessary lin+} of Assumption \ref{ass:msr convergence} are satisfied 
(replacing $\theta_{\tau, \epsilon}$ with $\theta$). 
Then $(\mu_k)_{k\in\mathbb{N}}$ converges gauge monotonically with respect to  
$d_{W_{2,\pbf}}$ to 
some $\pi^{\mu_0}\in\inv\Pcal\cap\mathscr{P}_2(G)$ with rate $O(\theta^{(k)}(t_0)$ 
if Assumption \ref{ass:msr convergence}\eqref{t:msr convergence, necessary sublin} holds, and 
in the case of Assumption \ref{ass:msr convergence}\eqref{t:msr convergence, necessary lin+} 
with rate $O(s_k(t_0))$ for 
$s_k(t)\equiv 
\sum_{j=k}^\infty \theta^{(j)}(t)$ and $t_0\equiv d_{W_{2,\pbf}}(\mu_0, \inv\Pcal)$.  
Moreover, $\Supp\pi^{\mu_0}\subset C$ for $C$ defined by \eqref{eq:stoch_feas_probl}.
\end{propn}
\begin{proof}
In both cases, the proof of convergence with the respective rates follows exactly the proof of 
convergence in Theorem \ref{t:msr convergence}.   
For the last statement, Theorem \ref{t:invMeasforParacontra}\eqref{t:invMeasforParacontra ii}  establishes that 
$\Supp\pi^{\mu_0}\subset C$, which completes the 
proof.
\end{proof}

The following is a generalization of \cite[Theorem 3.15]{HerLukStu19a}.
\begin{thm}[necessity of metric subregularity for monotone sequences]
\label{t:msr necessary}
Let $G\subset\Ecal$ be compact.  Let the Markov operator corresponding to 
Algorithm \eqref{algo:sbi}, $\map{\Pcal}{\mathscr{P}_2(G)}{\mathscr{P}_2(G)}$,
be a paracontraction with respect to the weighted Wasserstein metric $d_{W_{2,\pbf}}$.  
Suppose  
all sequences  
$(\mu_k)_{k\in\mathbb{N}}$ corresponding to Algorithm \ref{algo:sbi}
and initialized in $\mathscr{P}_2(G)$ are gauge monotone 
relative to $\inv\Pcal$ 
with rate $\theta$ satisfying \eqref{eq:theta} and at least one of the conditions in 
Assumption \ref{ass:msr convergence}.  Suppose, in addition, that  
$(\Id - \theta)^{-1}(\cdot)$ is continuous on $\Rbb_+$, strictly increasing, 
and $(\Id - \theta)^{-1}(0)=0$.  Then $\Psi$ defined by 
\eqref{eq:Psi} is gauge metrically subregular for $0$ relative to 
$\mathscr{P}_2(G)$ on $\mathscr{P}_2(G)$ 
with gauge $\rho(\cdot)=(\Id-\theta)^{-1}(\cdot)$, i.e. $\Psi$ satisfies 
\eqref{e:metricregularity}.
\end{thm}
\begin{proof}
If the sequence $(\mu_k)_{k\in\Nbb}$ is gauge monotone
relative to $\inv\Pcal$ 
with rate $\theta$ satisfying \eqref{eq:theta} and at least one of the conditions in 
Assumption \ref{ass:msr convergence}, 
then by the triangle inequality 
\begin{eqnarray}
d_{W_{2,\pbf}}(\mu_{k+1},\mu_k)&\geq& d_{W_{2,\pbf}}(\mu_k, \mubar_{k+1}) - d_{W_{2,\pbf}}(\mu_{k+1}, \mubar_{k+1}) \nonumber\\
&\geq& d_{W_{2,\pbf}}(\mu_k, \inv\Pcal) - d_{W_{2,\pbf}}(\mu_{k+1}, \inv\Pcal) \nonumber\\
&\geq& d_{W_{2,\pbf}}(\mu_k, \inv\Pcal) - \theta\paren{d_{W_{2,\pbf}}(\mu_k, \inv\Pcal)}\geq 0 \quad  \forall k\in \Nbb,
\label{e:Robin}
\end{eqnarray}
where $\mubar_{k+1}$ is a metric projection of $\mu_{k+1}$ onto $\inv\Pcal$ (exists since $\inv\Pcal$ is closed
in $\mathscr{P}_2(G)$).  
On the other hand, by Theorem \ref{t:invMeasforParacontra}\eqref{t:invMeasforParacontra ii}, 
inequality \eqref{eq:Psi_StoPBForBS} is tight, so $\Psi^{-1}(0) = \inv\Pcal$ and 
\begin{eqnarray}
\Psi(\mu_k) = \paren{\int_G\|x-T_1x\|^2\mu_k(dx)}^{1/2} &\geq&  
\inf_{\gamma\in C(\mu_k\Pcal, \mu_k)}\paren{\int_{G\times G}\|x-y\|_\pbf^2\gamma(dx,dy)}^{1/2}\nonumber\\
&=& d_{W_{2,\pbf}}(\mu_{k+1}, \mu_k) 
\quad\forall k\in\Nbb.\label{e:Hood}
\end{eqnarray}
Combining \eqref{e:Robin} and \eqref{e:Hood} yields
\begin{equation}\label{e:dumb}
d(0,\Psi(\mu_k)) = \Psi(\mu_k)\geq d_{W_{2,\pbf}}(\mu_k, \Psi^{-1}(0)) - 
\theta\paren{d_{W_{2,\pbf}}(\mu_k, \Psi^{-1}(0))} \quad  \forall k\in \Nbb.
 \end{equation}
By assumption $(\Id - \theta)^{-1}(\cdot)$ is continuous on $\Rbb_+$, 
strictly increasing, 
and $(\Id - \theta)^{-1}(0)=0$, so 
\begin{equation}\label{e:dumber}
(\Id - \theta)^{-1}\paren{d(0,\Psi(\mu_k))}\geq d_{W_{2,\pbf}}(\mu_k, \Psi^{-1}(0)) 
\quad  \forall k\in \Nbb.
 \end{equation}
Since this holds for {\em any} sequence $(\mu_k)_{k\in\Nbb}$ initialized in $\mathscr{P}_2(G)$
and these converge by Proposition \ref{t:rm and qafne to convergence} to points in $\inv\Pcal\cap\mathscr{P}_2(G)$, 
we conclude that $\Psi$ is metrically subregular for $0$ on $\mathscr{P}_2(G)$ with 
gauge $\rho=(\Id - \theta)^{-1}$.  
\end{proof}

\section{Block-Stochastic Splitting for Composite Optimization}\label{s:sbs}
We return now to stochastic blockwise  methods for solving \eqref{e:P1}.  
It is already understood that the critical points of $f + \sum_{j=1}^m g_j$, 
denoted $\crit(f + \sum_{j=1}^m g_j)$, are 
fixed points of the deterministic, non-block versions of Algorithms \ref{algo:sblfb} and 
\ref{algo:sbdr}; and fixed points of the deterministic, non-block versions of these algorithms 
are invariant distributions corresponding to iterates of these same stochastic blockwise algorithms.  
When $\inv \Pcal = \mathscr{C}$ defined by 
\eqref{eq:stoch_feas_probl2}, then in fact  
any $\xbar\in C\equiv \set{x}{\mathbb{P}(x\in\Fix T_\xi)=1}$
is almost surely at least a stationary point.  
This leads to the following elementary observations. 
\begin{lem}\label{t:inv P - crit pts}
 Let $T_i$ defined by either \eqref{e:TFB_i} (if $f$ is differentiable) or \eqref{e:TDR_i} 
 be single-valued on $\Ecal$ and let 
 $\Pcal$ be the Markov operator with update function $T_i$.  Then 
 $\crit(f + \sum_{j=1}^m g_j)\subset S\equiv \bigcup_{\pi\in\inv\Pcal}\Supp\pi$, and 
 if $f$ and $g_j$ ($j=1,\dots, m$) are convex, then 
 $\xbar\in C$ if and only if $\xbar\in \crit(f + \sum_{j=1}^m g_j)$ almost surely.  
\end{lem}

\subsection{Regularity}
In this section we determine the regularity of the blockwise mappings $T_i$ for the two cases 
\eqref{e:TFB_i} and \eqref{e:TDR_i}.  In Theorem \ref{t:afneie}, the regularity constants $\epsilon_i$
and $\alpha_i$ are bounded above by the constants of $T_1$, which is the mapping including
all of the blocks.  It suffices, then, to determine the regularity of $T_1$ for the two cases 
\eqref{e:TFB_i} and \eqref{e:TDR_i}.

\begin{prop}[regularity of partial resolvents]\label{t:f_j}
    For $j=1,2,\dots, m$, for each vector of parameters $x\in G\subset \Ecal$, let 
    $\map{f_j(\cdot;x)}{G_j\subset \Ecal_j}{\extre}$ defined by \eqref{e:f_j} be subdifferentially regular 
    with subdifferentials satisfying 
    \begin{eqnarray}\exists \tau_{f_j}\geq 0: &&\forall x\in G, \forall u_j, v_j\in  G_j, 
    ~\forall z_j\in t_j\sd f_j(u_j; x), w_j\in t_j\sd f_j(v_j; x), \nonumber \\
        &&-\tfrac{\tau_{f_j}}{2}\norm{(u_j+z_j)
        -(v_j+w_j)}^2\nonumber\\
        &&\qquad\qquad \qquad\qquad \leq \ip{z_j-w_j}{u_j-v_j}.
        \label{e:sd f submon}
    \end{eqnarray}%
    For $f_t(u; x)\equiv \sum_{j=1}^m t_j f_j(u_j;x)$, the resolvent 
    $J_{\partial f_t, 1}$ is a$\alpha$-fne with constant 
    $\alpha_f=1/2$ and violation $\tau_f = \max_j\{\tau_{f_j}\}$ on $G$.  
    If $f_j$ is convex on $\Ecal_j$ for each $j=1,2,\dots,m$, 
    then $J_{\partial f_t, 1}$ is $\alpha$-fne with constant $\alpha_f=1/2$ and no violation on $\Ecal$. 
\end{prop}
\noindent Condition  \eqref{e:sd f submon} generalizes the notion of hypomonotonicity \cite{VA}
and is satisfied by any {\em prox-regular} function.  

\begin{proof}
    By  \cite[Proposition 2.3(iv)]{LukNguTam18}, condition \eqref{e:sd f submon} 
    is equivalent to $J_{\partial f_j, t_j}$ being 
    a$\alpha$-fne on $G_j$ with constant $\alpha_{f_j} = 1/2$ and 
    violation $\tau_{f_j}$.  
    Extending this, for $f_t(x) \equiv \sum_{j=1}^m t_jf_j(u_j;x)$ we have
    $\sd f_t(u; x) = \ecklam{\ecklam{t_1\sd_{u_1}f_1(u_1; x)}^T,\dots,\ecklam{t_m\sd_{u_m}f_m(u_m; x)}^T }^T$ and 
    \begin{eqnarray*}&&\forall v,u\in \subset G,\mbox{ for } z\equiv\sd f_t(u;x), w\equiv\sd f_t(v;x), \quad\nonumber\\ 
    &&\ip{z-w}{u-v}=
        \sum_{j=1}^m\ip{z_j-w_j}{u_j-v_j}_{G_j}\nonumber\\
        &&\qquad\qquad\qquad\quad\geq\sum_{j=1}^m\tfrac{-\tau_{f_j}}{2}\norm{(u_j+z_j)-(v_j+w_j)}_{G_j}^2\nonumber\\
        &&\qquad\qquad\qquad\quad\geq\tfrac{-\max_{j}\{\tau_{f_j}\}}{2}\sum_{j=1}^m\norm{(u_j+z_j)-(v_j+w_j)}_{G_j}^2\nonumber\\
        &&\qquad\qquad\qquad\quad=\tfrac{-\tau_{f}}{2}\norm{(u+z)-(v+w)}^2.
    \end{eqnarray*}
    Application of \cite[Proposition 2.3(iv)]{LukNguTam18} to $f_t$ establishes the claim. The convex statement
    follows from monotonicity of the gradient.  
\end{proof}

\noindent The following corollary is just the specialization of Proposition \ref{t:f_j} to the 
case that $f_j(\cdot;x)$ is independent of the parameter $x$.  
\begin{cor}[regularity of resolvents of block separable functions]\label{t:h_j}
    In the setting of Proposition \ref{t:f_j} let $\map{h_j(\cdot)}{G_j\subset \Ecal_j}{\extre}$ satisfy
    \begin{eqnarray}\exists \tau_{h_j}\geq 0: &&\forall x_j, y_j\in  \Ecal_j, 
    ~\forall z_j\in t_j\sd h_j(x_j), w_j\in t_j\sd h_j(y_j), \nonumber \\
        \label{e:sd g submon}&&-\tfrac{\tau_{h_j}}{2}\norm{(x_j+z_j)-(y_j+w_j)}^2\leq \ip{z_j-w_j}{x_j-y_j}.
    \end{eqnarray}%
    Then for $h_t(x)\equiv \sum_{j=1}^m t_j h_j(x_j)$, the resolvent $J_{\partial h_t, 1}$ 
    is a$\alpha$-fne with constant $\alpha_h=1/2$ and violation $\tau_h = \max_j\{\tau_{h_j}\}$ on $G$.  
    If $h_j$ is convex on $\Ecal_j$ for each $j=1,2,\dots,m$, 
    then $J_{\partial h_t, 1}$ is $\alpha$-fne with constant $\alpha_h=1/2$ and no violation on $\Ecal$. 
\end{cor}
\begin{prop}[regularity of gradient descent]\label{t:T_GD}
		Let $\map{f}{\Ecal}{\Rbb}$ be continuously differentiable with blockwise Lipschitz and 
		hypomonotone gradient, that is $f$ satisfies 
		\begin{subequations}\label{e:f reg}
		 \begin{eqnarray}
            \forall j=1,2,\dots,m, ~\exists L_j>0:\quad 
            \sum_{j=1}^m\|\nabla_{x_j} f(x) - \nabla_{x_j} f(y)\|^2&\leq& \sum_{j=1}^m L^2_j\|x_j-y_j\|^2\nonumber\\
            &&\forall x,y\in\Ecal,
            \label{e:f grad-Lip}
		 \end{eqnarray}
        and 
        \begin{eqnarray} 
            \forall j=1,2,\dots,m, ~\exists \tau_{f_j}\geq 0:\sum_{j=1}^m-\tau_{f_j}\norm{x_j-y_j}^2&\leq& 
            \sum_{j=1}^m\ip{\nabla_{x_j} f(x)-\nabla_{x_j} f(y)}{x_j-y_j}\nonumber\\
            &&
            \qquad\qquad\qquad\qquad\qquad\forall x,y\in\Ecal.
            \label{e:hypomonotone}
        \end{eqnarray}
		\end{subequations}
		Then the gradient descent mapping with blockwise heterogeneous
		step lengths defined by\\
		$T_{GD}\equiv \Id - \bigoplus_{j=1}^m t_j\nabla_{x_j} f$ 
is a$\alpha$-fne on $\Ecal$ with violation at most 
\begin{subequations}
\begin{equation}\label{e:nabla f aalph-fne violation}
\epsilon_{GD} = \max_{j}\left\{2t_j\tau_j + \tfrac{t_j^2L_j^2}{\alphabar}\right\}<1, 
\quad\mbox{ with constant }\quad  \alphabar=\max_j\{\alpha_j\}
\end{equation} 
whenever the blockwise steps $t_j$ satisfy 
\begin{equation}\label{e:nabla f aalph-fne step}
t_j\in\paren{0,\frac{\alphabar\sqrt{\tau_j^2+ L_j^2} - \alphabar\tau_j}{L_j^2}}. 
\end{equation} 
\end{subequations}

If $f$ is convex then, with global step size 
$t_j = t < \tfrac{2\alphabar}{\Lbar}$ ($j=1,2,\dots,m$) 
for $\alphabar\in (0,1)$ with $\Lbar =\max_j\{L_j\}$,
the gradient descent mapping $T_{GD}$ is $\alpha$-fne with 
constant $\alphabar$ (no violation). 
\end{prop}
\begin{proof}
 By \cite[Proposition 2.1]{LukNguTam18}, the claim holds  
if and only if $\Id - \tfrac{1}{\alphabar}\bigoplus_{j=1}^m t_j\nabla_{x_j} f$ is almost nonexpansive
on $\Ecal$ with violation at most 
\begin{equation}\label{e:nabla f_i ane eps}
 \epsilon' = \epsilon_{GD}/\alphabar = 
 \tfrac{1}{\alphabar}\max_{j}\left\{2t_j\tau_j + \tfrac{t_j^2L_j^2}{\alphabar}\right\}.
\end{equation}
To see this latter property, since $f$ satisfies \eqref{e:f reg} we have 
		\begin{eqnarray}
   			&&\norm{\paren{x - \tfrac{1}{\alphabar}\bigoplus_{j=1}^m t_j\nabla_{x_j} f(x)} - 
   			\paren{y - \tfrac{1}{\alphabar}\bigoplus_{j=1}^m t_j\nabla_{x_j} 
   			f\paren{y}}}^{2}  \nonumber\\
   			&&\qquad = 
   			\norm{x - y}^{2}- \tfrac{2}{\alphabar}\sum_{j=1}^m t_j\ip{x_j - y_j}%
   			{\nabla_{x_j} f(x) - \nabla_{x_j} f\paren{y}} + \tfrac{1}{\alphabar^2}\sum_{j=1}^m%
   			t_j^2\norm{\nabla_{x_j} f(x) - \nabla_{x_j} f\paren{y}}^{2} \nonumber \\
   			&&\qquad \leq 
            \norm{x - y}^{2}+ \tfrac{2}{\alphabar}\sum_{j=1}^mt_j\tau_j\norm{x_j - y_j}^2 + 
            \tfrac{1}{\alphabar^2}\sum_{j=1}^m%
   			t_j^2L_j^2\norm{x_j - y_j}^{2} \nonumber \\
			&&\qquad  \leq \paren{1 + 
			\tfrac{1}{\alphabar}\max_{j}\left\{2t_j\tau_j + \tfrac{t_j^2L_j^2}{\alphabar}\right\}}
			\norm{x - y}^{2} 
			\label{e:interim1}
		\end{eqnarray}
		for all $x,y\in \Ecal$.  A simple calculation shows that the violation does not exceed $1$ whenever the step $t_j$ is bounded by \eqref{e:nabla f aalph-fne step}.  This proves the result for 
		the nonconvex setting. 
		
		If $f$ is convex, then \cite[Proposition 3.4]{Kartamyschew} shows that a different bound on the steps is possible.  
		Note that 
		by \cite[Corollaire 10]{BaiHad77}
		\begin{eqnarray*}
            \tfrac{1}{L}\|\nabla f(x) - \nabla f(y)\|^2&\leq& 
                \ip{\nabla f(x) - \nabla f(y)}{x-y}
		\end{eqnarray*}
		Let $\alphabar=\max_j\{\alpha_j\}$ with $\alpha_j\in (0,1)$ and $\Lbar =\max_j\{L_j\}$.  
		For $t = \tfrac{2\alphabar}{\Lbar}$  we have 
		$2t = \tfrac{t^2\Lbar}{\alphabar}$ and 
		\begin{eqnarray*}
            &\tfrac{t^2\Lbar}{\alphabar}\tfrac{1}{\Lbar}\|\nabla f(x) - \nabla f(y)\|^2\leq 
                2t\ip{\nabla f(x) - \nabla f(y)}{x-y}&\\
                &\iff&\\
            &\tfrac{1}{\alphabar}\|t\nabla f(x) - t\nabla f(y)\|^2\leq 
                2\ip{t\nabla f(x) - t\nabla f(y)}{x-y}&\\
               &\iff&\\
            &\norm{x-y}^2 + 
            \paren{1+\tfrac{1-\alphabar}{\alphabar}}\|t\nabla f(x) - t\nabla f(y)\|^2&\\
            &\qquad \leq 
                 2\ip{t\nabla f(x) - t\nabla f(y)}{x-y} + \norm{x-y}^2&\\
               &\iff&\\
            &\norm{\paren{x - t\nabla f(x)} - \paren{y - t\nabla f(y)}}^2 
             &\\
            &\qquad \leq \norm{x-y}^2 - \tfrac{1-\alphabar}{\alphabar}\|t\nabla f(x) - t\nabla f(y)\|^2&\\
               &\iff&\\
            &\norm{\paren{x - \bigoplus_{j=1}^m t\nabla_{x_j} f(x)} - 
            \paren{y - \bigoplus_{j=1}^m t\nabla_{x_j} f(y)}}^2 
             &\\
            &\qquad \leq \norm{x-y}^2 - \sum_{j=1}^m\tfrac{1-\alphabar}{\alphabar}
            \|t\nabla_j f(x) - t\nabla_j f(y)\|^2&\\
               &\iff&\\
            &\norm{T_{GD}x - T_{GD}y}^2 
            \leq \norm{x-y}^2 - \tfrac{1-\alphabar}{\alphabar}\psi(x, y, T_{GD}x, T_{GD}y)&           
        \end{eqnarray*}
where the last implication follows from \eqref{eq:nice ineq} with blockwise step $t_j=t$ for all 
$j$ in $T_{GD}$ .

\end{proof}

\begin{remark}
 The violation in the nonconvex case can be controlled by choosing a smaller blockwise step $t_j$.  
 In the convex setting, larger step sizes are possible, but these are limited by the global
 Lipschitz constant $\Lbar$ and the constant $\alphabar$.  Note that the upper bound on the 
 step length suggested by Proposition \ref{t:T_GD} is consistent with the upper bound on the steps in Example \ref{eg:T_i not afne}.  
\end{remark}
\begin{propn}[blockwise composite mappings]\label{t:fb-dr aafne}$~$
Let $G\subset\Ecal$ with $G_j\subset\Ecal_j$ for $j=1,2\,\dots,m$. 
	\begin{enumerate}[(i)]
	\item\label{t:fb-dr aafne ncvx} \textbf{Fully nonconvex.}
		For all $j\in\{1,2,\dots,m\}$ let $\map{f}{G}{\Rbb}$ be
		subdifferentially regular with subdifferential satisfying \eqref{e:sd f submon} and
		let $\map{h_j}{G_j}{\extre}$ be proper, l.s.c., 
		and subdifferentially regular satisfying \eqref{e:sd g submon}.
		\begin{enumerate}[(a)]
          \item\label{t:dr aafne ncvx} The \textit{partial blockwise Douglas-Rachford} mapping $T^{DR}_{i}$ 
          defined by \eqref{e:TDR_i} 
          ($j\in M_i$) is 
          a$\alpha$-fne on $G_{M_i}\bigoplus \{z\}_{M_i^\circ} $
          for any fixed $z_{M_i^\circ}\in G_{M_i^\circ}$ with respective constant and violation 
          \begin{equation}\label{e:alpha_DR}
            \alpha_{DR}= \frac{2}{3}, \mbox{ and }\quad
            \epsilon_{DR} \leq \tau_f+\tau_{h} + \tau_f\tau_{h}
          \end{equation}
          where $\tau_{h} \equiv \max_j\{\tau_{h_j} \}$
          and $\tau_{f} \equiv \max_j\{\tau_{f_j} \}$.
          \item\label{t:fb aafne ncvx} If $f$ is continuously differentiable on $\Ecal$ and satisfies 
          \eqref{e:f reg}, the \textit{partial blockwise forward-backward} mapping $T^{FB}_{i}$ 
          defined by \eqref{e:TFB_i} with step lengths   
          $t_j$ satisfying \eqref{e:nabla f aalph-fne step} ($j\in M_i$) is 
          a$\alpha$-fne on affine subspaces $G_{M_i}\bigoplus \{z\}_{M_i^\circ} $
          for any fixed $z_{M_i^\circ}\in G_{M_i^\circ}$ with respective constant and violation 
          \begin{equation}\label{e:alpha_FB}
            \alpha_{FB}\equiv \frac{2}{1+\tfrac{1}{\max\{\tfrac{1}{2},~ \alphabar \}}}, \mbox{ and }\quad
            \epsilon_{FB} \leq \epsilon_{GD}+\tau_{h} + \epsilon_{GD}\tau_{h}
          \end{equation}
          where $\alphabar \equiv \max_j\{\alpha_j\}$, $\tau_{h} \equiv \max_j\{\tau_{h_j} \}$
          and $\epsilon_{GD}$ is no larger than \eqref{e:nabla f aalph-fne violation}.
          \end{enumerate}
	\item\label{t:fb-dr aafne ncvx-cvx}\label{t:fb averaged cvx-ncvx} \textbf{Partially nonconvex.}  
      For all $j\in\{1,2,\dots,m\}$ let $\map{f}{\Ecal}{\Rbb}$ be
		continuously differentiable with gradient satisfying \eqref{e:f reg} and
        let the functions $h_j$ be convex on $G_j$  
	($j=1,2,\dots,m$).  Then for all $i\in\Ibb$, $T^{FB}_i$ 
        is a$\alpha$-fne on $G_{M_i}\bigoplus \{z\}_{M_i^\circ}$ for any 
        $z_{M_i^\circ}\in G_{M_i^\circ}$ 
        with constant $\alpha_{FB}$ given by \eqref{e:alpha_FB}, violation  $\epsilon_{FB}$  
        at most $\epsilon_{GD}$, and this can be made arbitrarily small by choosing the 
        step lengths $t_i$ small enough.  
    \item\label{t:fb-dr aafne cvx} \textbf{Convex.}  If $f$ and $h_j$ are convex on $\Ecal$ 
    ($j=1,2,\dots,m$), then
      \begin{enumerate}[(a)]
        \item $T^{DR}_i$ is $\alpha$-fne on $\Ecal_{M_i}\bigoplus \{z\} $ with 
        constant $\alpha_{DR}=2/3$ and no violation;
        \item if $f$ is continuously differentiable  and $\nabla f$ satisfies \eqref{e:f grad-Lip}, 
        $T^{FB}_i$ with global step size $t<\tfrac{2\alphabar}{\Lbar}$ for $\Lbar=\max_j\{L_j\}$ is $\alpha$-fne 
        on $\Ecal_{M_i}\bigoplus \{z\} $ with 
        constant $\alpha_{FB}$ given by \eqref{e:alpha_FB} and no violation.
        \end{enumerate}
	\end{enumerate}	
	\end{propn}	
\begin{proof}
Part \eqref{t:fb-dr aafne ncvx}.   
By Theorem \ref{t:afneie}, the respective regularity constants $\epsilon_i$
and $\alpha_i$ are bounded above by the respective constants of $T^{FB}_1$ and $T^{DR}_1$, 
which are the mappings including
all of the blocks.  It suffices, then, to determine the regularity of $T^{FB}_1$ and $T^{DR}_1$.
Part \eqref{t:dr aafne ncvx}. By Proposition \ref{t:f_j} and Corollary \ref{t:h_j} 
$J_{\partial f_t}$ and $J_{\partial h_t}$ are a$\alpha$-fne 
with constant $\alpha_{h_t} = 1/2$ and violation $\tau_{f} = \max_j\{\tau_{f_j}\}$ 
(respectively $\tau_{h} = \max_j\{\tau_{h_j}\}$) on $G$.  
Then by 
\cite[Proposition 2.4]{LukNguTam18} $T_1^{DR}$ is a$\alpha$-fne with constant $\alpha_{DR}=2/3$
and (maximal) violation 
given by \eqref{e:alpha_DR} on $G$.

Part \eqref{t:fb aafne ncvx}. By Proposition \ref{t:T_GD}, $T_{GD}$ is 
a$\alpha$-fne on $G$ with violation $\epsilon_{GD}$ no larger than 
\eqref{e:nabla f aalph-fne violation}
and constant $\alphabar=\max_j\{\alpha_j\}$.  By Corollary \ref{t:h_j} $J_{\partial h_t}$ is a$\alpha$-fne 
with constant $\alpha_{h_t} = 1/2$ and violation $\tau_{h} = \max_j\{\tau_{h_j}\}$ on $\Ecal$.  Then by 
\cite[Proposition 2.4/Proposition 3.7]{LukNguTam18} $T_1^{FB}$ is a$\alpha$-fne with constant $\alpha_{FB}$
and (maximal) violation 
given by \eqref{e:alpha_FB} on $G$.

Parts \eqref{t:fb-dr aafne ncvx-cvx}-\eqref{t:fb-dr aafne cvx} follow immediately from 
part \eqref{t:fb-dr aafne ncvx} and Propositions \ref{t:f_j}-\ref{t:T_GD}. 
\end{proof}

\begin{cor}\label{t:aafneie fb}  For $G\subset\Ecal$, let $\map{\Phi}{G\times \Ibb}{G}$ be 
the update function given by 
$\Phi(x,i) = T_{i}x$ where $T_i$ is either $T^{FB}_i$ or $T^{DR}_i$ defined respectively 
by \eqref{e:TFB_i} and \eqref{e:TDR_i}.
\begin{enumerate}[(i)]
 \item \textbf{Fully nonconvex.}  
 Under the assumptions of Proposition \ref{t:fb-dr aafne}\eqref{t:fb-dr aafne ncvx}, 
    that is both $f$ and $h$  in \eqref{e:TFB_i} are nonconvex, the corresponding 
    update function $\Phi(x,i)$ is 
    a$\alpha$-fne in expectation with respect to the weighted norm $\|\cdot\|_\pbf$ 
    with regularity constants $\pbar\epsilon_{DR}$ and $\alpha_{DR}$ 
    (respectively $\pbar\epsilon_{FB}$ and $\alpha_{FB}=2/3$) 
    corresponding to \eqref{e:alpha_DR} (respectively \eqref{e:alpha_FB}).
\item \textbf{Partially nonconvex.}
Under the assumptions of Proposition \ref{t:fb-dr aafne}\eqref{t:fb-dr aafne ncvx-cvx},  
    that is $f$ smooth nonconvex with Lipschitz and hypomonotone gradient and $h_j$ convex in \eqref{e:TFB_i},
    $\Phi(x,i)=T^{FB}_i(x)$ is a$\alpha$-fne in expectation with respect to the weighted norm 
    $\|\cdot\|_\pbf$ with constant $\alpha_{FB}$ as above 
    and violation at most $\pbar\epsilon_{GD}$ with $\epsilon_{GD}$ 
    given by \eqref{e:nabla f aalph-fne violation};
    this violation can be made arbitrarily small by choosing the 
    step lengths $t_i$ small enough.  
\item \textbf{Convex.} If both $f$ and $h_j$ 
    are convex on $\Ecal$ ($j=1,2,\dots,m$), then  
    $T^{DR}_i(x)$ 
    is $\alpha$-fne in expectation with respect to the weighted norm 
    $\|\cdot\|_\pbf$  (no violation) and constant $\alpha_{DR}=2/3$ on $\Ecal$. 
    In the case of $T^{FB}$, if $\nabla f$ satisfies \eqref{e:f grad-Lip} and the 
    global step size is bounded by $t<\tfrac{2\alphabar}{\Lbar}$ for 
    $\Lbar=\max_j\{L_j\}$, 
    $T^{FB}_i(x)$ 
    is $\alpha$-fne in expectation with respect to the weighted norm 
    $\|\cdot\|_\pbf$  (no violation) and constant $\alpha_{GD}$ on $\Ecal$. 
\end{enumerate}
\end{cor}
\begin{proof}
 This is an immediate consequence of Proposition \ref{t:fb-dr aafne} and 
 Theorem \ref{t:afneie}
\end{proof}

Before presenting the convergence results it is worthwhile pointing out that 
the partial blockwise forward-backward mappings $T^{FB}_i$ and $T^{DR}_i$
have common fixed points, and these are critical points of 
\eqref{e:P1}.  In other words, the stochastic fixed point problem is consistent. 
As shown in Section \ref{s:consistent2}, in this case the metric subregularity condition 
\eqref{e:metricregularity} simplifies to \eqref{eq:Sport day} when 
$\Psi(\mu)=0$ if and only if $\mu\in\inv\Pcal$ 
and  $\Supp\mu\subset C$.  In the convex setting we have the following correspondence
between invariant measures of the stochastic block iterations and minima of \eqref{e:P1}.
\begin{prop}
 Let $\Pcal$ be the Markov operator associated with either Algorithm \ref{algo:sblfb} or 
 \ref{algo:sbdr}. In the setting of Lemma \ref{t:inv P - crit pts}, if $f$ and $g_j$ 
 (for all $j=1,\dots,m$) are convex, 
 then $\inv\Pcal = \set{\pi}{\Supp\pi\subset C}$ and whenever $x\in C$ then 
  almost surely $x\in\argmin\{ f +\sum_{j=1}^m g_j\}$.
\end{prop}
\begin{proof}
When $f$ and $g_j$ (for all $j=1,\dots,m$) are convex, the corresponding mappings $T_i$ 
defined by either \eqref{e:TFB_i} or \eqref{e:TDR_i} are single-valued self-mappings on $\Ecal$ 
and $\alpha$-fne on $\Ecal_{M_i}\oplus \{z\}_{M^\circ_i}$ for every $z\in M_i^\circ$ 
as long as the step size $t_i$ is small enough
(\cite[Propositions 3.7 and 3.10]{LukNguTam18} specialized to the convex case).  Then
by Corollary \ref{t:para} the mappings $T_i$ are  
paracontractions in expectation on $\Ecal$.  
 The claim then follows from Theorem \ref{t:invMeasforParacontra}\eqref{t:invMeasforParacontra ii} 
 and Lemma \ref{t:inv P - crit pts} since in this case 
 $\crit(f + \sum_{j=1}^m g_j)= \argmin\{ f +\sum_{j=1}^m g_j\}$.  
\end{proof}

The final result of this study collects all of these facts in the context of the 
Markov chain underlying Algorithm \ref{algo:sblfb} and \ref{algo:sbdr}. 
\begin{prop}\label{t:convergence sbfb-sbdr}
Let $\Pcal$ be the Markov operator associated with the S-BFBS Algorithm \ref{algo:sblfb}
or the S-BDRS Algorithm \ref{algo:sbdr} and let $\paren{\mu_{k}}_{k\in\Nbb}$ be the 
corresponding sequence of measures initialized by any $\mu_0\in\mathscr{P}_2(G)$, 
where $G$ is a closed subset of $\Ecal$.  
Assume that 
$G\supset\crit\paren{f +\sum_{j=1}^m g_j}\neq\emptyset$ and $\inv\Pcal = \mathscr{C}$
defined by \eqref{eq:stoch_feas_probl2}.
Let $\Psi$ given by \eqref{eq:Psi} be such that  
$\Psi(\mu)=0$ if and only if $\mu\in\inv\Pcal$.  
Additionally, let the mappings $T_i$ be self-mappings on  
$G$ where $\Psi$  satisfies 
\eqref{eq:Sport day} with gauge $\rho$ given by \eqref{eq:gauge} 
  with $\tau=(1-\alpha_{*})/\alpha_{*}$, $\epsilon = \pbar\epsilon_{*}$ 
  for constants $\alpha_{*}$ and violation $\epsilon_{*}$ given by  
  either \eqref{e:alpha_FB} or \eqref{e:alpha_DR} (depending on the algorithm), 
  and  $\theta_{\tau,\epsilon}$ satisfying \eqref{eq:theta_tau_eps}.
\begin{enumerate}
 \item   \textbf{Fully nonconvex.}  Under the assumptions of Proposition 
 \ref{t:fb-dr aafne}\eqref{t:fb-dr aafne ncvx},  the sequence $(\mu_{k})$ satisfies 
 \[
  d_{W_{2,\pbf}}\paren{\mu_{k+1},\inv\mathcal{P}}
                \leq \theta_{\tau,\epsilon}\paren{d_{W_{2,\pbf}}\paren{\mu_k,\inv\mathcal{P}}} 
                \quad \forall k \in \mathbb{N}.  
 \]
   If $\tau$ and $\epsilon$ are such that at least one of the conditions in Assumption \ref{ass:msr convergence} holds, 
then $\mu_k\to \pi^{\mu_0}\in\inv\mathcal{P}\cap\mathscr{P}_2(G)$
 in the $d_{W_{2,\pbf}}$ metric with rate 
 $O\paren{\theta_{\tau, \epsilon}^{(k)}(t_0)}$ in the case of 
 Assumption \ref{ass:msr convergence}\eqref{t:msr convergence, necessary sublin}
 where $t_0 = d_{W_{2,\pbf}}\paren{\mu^0, \inv\Pcal}$,  and with rate  $O(s_k(t_0))$ for 
$s_k(t)\equiv 
\sum_{j=k}^\infty \theta_{\tau, \epsilon}^{(j)}(t)$
in the case of Assumption \ref{ass:msr convergence}\eqref{t:msr convergence, necessary lin+}.  
Moreover, 
$\Supp\pi^{\mu_0}\subset C\equiv \set{x\in G}{\Pbb\paren{T_\xi x=x}=1}$.
 \item \textbf{Partially nonconvex.} 
 Under the assumptions of Proposition 
 \ref{t:fb-dr aafne}\eqref{t:fb-dr aafne ncvx-cvx},   
  that is $f$ smooth nonconvex with Lipschitz and hypomonotone gradient and $h_j$ convex in \eqref{e:TFB_i}, 
  if there exist $\tau$ and $\epsilon$ such that at least one of the conditions in Assumption \ref{ass:msr convergence} holds, then for 
 all step lengths $t_i$ small enough in  \eqref{e:TFB_i}
 and any initial distribution 
 $\mu_0$ close enough to $\inv\Pcal$, the sequence 
 $\mu_k\to \pi^{\mu_0}\in\inv\mathcal{P}\cap\mathscr{P}_2(G)$
 in the $d_{W_{2,\pbf}}$ metric with rate at least
 $O\paren{\theta_{\tau, \epsilon}^{(k)}(t_0)}$ in the case of 
 Assumption \ref{ass:msr convergence}\eqref{t:msr convergence, necessary sublin},  and with rate  at least $O(s_k(t_0))$ in the case of Assumption \ref{ass:msr convergence}\eqref{t:msr convergence, necessary lin+};  moreover, 
$\Supp\pi^{\mu_0}\subset C$.
 \item \textbf{Convex.} If $f$ and $h_j$ are convex on $\Ecal$,  
    and there exists $\tau$ such that at least one of the conditions in Assumption \ref{ass:msr convergence} holds when $\epsilon=0$, 
    then  the sequence $(\mu_k)$ corresponding to Algorithm \ref{algo:sbdr} 
    initialized from any $\mu_0\in\mathscr{P}_2(\Ecal)$, 
    converges in the metric $d_{W_{2,\pbf}}$ to an 
    invariant distribution with rate at least
 $O\paren{\theta_{\tau, \epsilon}^{(k)}(t_0)}$ in the case of 
 Assumption \ref{ass:msr convergence}\eqref{t:msr convergence, necessary sublin},  and with rate  at least $O(s_k(t_0))$ in the case of Assumption \ref{ass:msr convergence}\eqref{t:msr convergence, necessary lin+}.
    Moreover 
    $\Supp\pi^{\mu_0}\subset C\equiv \argmin\paren{f+\sum_{j=1}^m g_j}$.
    If $f$ is continuously differentiable and satisfies \eqref{e:f grad-Lip}, then 
    the stated convergence in the case of Algorithm \ref{algo:sblfb} holds for the 
    global step length $t<\tfrac{2\alphabar}{\Lbar}$.  
    \end{enumerate}
\end{prop}

\section{Final Remarks}
There a several open technicalities lurking between the lines above, and one rather obvious
challenge hiding in plain sight.  To the hidden technicalities belong the question of
whether metric subregularity is necessary for quantitative convergence in some appropriate
metric of Markov operators that are not paracontractions in measure. We conjecture that 
this is true.  Another open technical issue concerns the statement of asymptotic regularity in 
Proposition \ref{t:asymp reg}.  This result is incomplete without some extension to a 
weak type of convergence in distribution.  For consistent 
stochastic fixed point problems, if each of the update functions
$T_i$ were $\alpha$-fne, then almost sure weak convergence of the iterates is 
guaranteed \cite[Theorem 3.9]{HerLukStu19a};  at issue here is whether this 
holds when $T_i$ is pointwise $\alpha$-fne {\em in expectation}
at invariant measures of the corresponding Markov operator.  We expect that there should be 
a counterexample to this claim.  Characterization of the supports of invariant measures 
in the inconsistent case is quite challenging and essential for meaningfully connecting the 
limiting distributions of the algorithms to solutions to the underlying optimization problem. 
Finally, the restriction of the study to single-valued mappings does not allow one to capture 
the full extent of behavior one sees with nonconvex problems.  Projection methods for sparse 
affine feasibility, for instance, have the property that the projection onto a sparsity constraint 
can be multi-valued on all neighborhoods of a solution (see \cite[Lemma III.2]{HesseLukeNeumann14}).  
An extension of the analysis presented here to multi-valued mappings, is required.

The most difficult challenge to all of this is the task of numerically
monitoring convergence in distribution of random variables.  To do this completely one needs first of all 
efficient means for computing the Wasserstein distance between measures;  in other words, one needs 
to solve optimal transport problems efficiently.  Again, for consistent stochastic feasibility when 
convergence of the iterates can be guaranteed almost surely, optimal transport is not needed;  more 
generally, however, this machinery is essential.  Secondly, one needs to numerically estimate the 
distributions whose distances are to be computed.  These are significant challenges worthy of attention.

\section{Funding and/or Conflicts of interests/Competing interests}
This work was supported in part by a grant from the Deutsche Forschungsgemeinschaft 
(DFG, German Research Foundation) – Project-ID 432680300 – SFB 1456.
The manuscript has not been submitted to any other journal for simultaneous consideration. 
The author has no financial or non-financial interests that are directly or indirectly related to the work 
submitted for publication. 


\end{document}